\documentclass[12pt]{amsart}
\usepackage{epsfig,psfrag,verbatim,amssymb,amsmath,amsfonts}
\newtheorem{theorem}{Theorem}
\newtheorem{lemma}[theorem]{Lemma}
\newtheorem{prop}[theorem]{Proposition}
\newtheorem{cor}[theorem]{Corollary}
\newtheorem{theo}{Theorem}
\newtheorem*{ex}{Example}
\newtheorem{rem}{Remark}
\newtheorem{defi}{Definition}

\newcommand{\field}[1]{\mathbb{#1}}

\newcommand{\EE}{\field{E}}

\newcommand{\Z}{\field{Z}}
\newcommand{\N}{\field{N}}

\newcommand{\PP}{\field{P}}

\newcommand{\wV}{\widetilde{V}}
\newcommand{\wZ}{\widetilde{Z}}
\newcommand{\wW}{\widetilde{W}}
\newcommand{\wK}{\overline{K}}
\newcommand{\wpi}{\overline{\pi}}
\newcommand{\wsi}{\overline{\si}}
\newcommand{\wrho}{\overline{\rho}}
\newcommand{\wzeta}{\overline{\zeta}}

\renewcommand{\tilde}{\widetilde}

\newcommand{\F}{\mathcal{F}}
\newcommand{\la} {\lambda}
\newcommand{\si}{\sigma}

\newcommand{\ga}{\gamma}
\newcommand{\om}{\omega}
\newcommand{\Om}{\Omega}
\newcommand{\eps}{\varepsilon}
\renewcommand{\kappa}{\varkappa}
\renewcommand{\epsilon}{\varepsilon}
\newcommand{\won}{{\boldsymbol 1}}
\newcommand{\cal}{\mathcal}

\newcommand{\geo}{\mbox{\rm Geom(1/2)} }
\newcounter{constante}
\newcommand{\con}[1]{
\immediate\write 1{\noexpand\newlabel{#1}{{\theconstante}{\theconstante}}}
                    c_{\theconstante}
                    \stepcounter{constante}
                   }

\begin{document}

\setcounter{page}{1}

\title[Excited random walks] {Positively and negatively excited random
  walks on integers, with branching processes}

\author{Elena Kosygina and Martin P.W.\ Zerner} \thanks{\textit{2000
    Mathematics Subject Classification.}  60K35, 60K37, 60J80.  }
\thanks{\textit{Key words:}\quad Central limit theorem, excited random
  walk, law of large numbers, positive and negative cookies,
  recurrence, renewal structure, transience.  }
  \thanks{Acknowledgment: E.\ Kosygina's work was partially supported by the CUNY Research Foundation, PSC-CUNY award
  \# 69580-00-38.}
\begin{abstract}
  We consider excited random walks on $\Z$ with a bounded number of
  i.i.d.\ cookies per site which may induce drifts both to the left
  and to the right.
   We extend the criteria for recurrence and transience by M.\,Zerner
  and for positivity of speed by A.-L.\,Basdevant and A.\,Singh to
  this case and also prove an annealed central limit theorem.
The proofs are based on results from the literature concerning
branching processes with migration and make use of a certain renewal
structure.
\end{abstract}
\maketitle

\section{Introduction}
We consider nearest-neighbor random walks on the one-dimensional
integer lattice in an i.i.d.\ cookie environment with a uniformly
bounded number of cookies per site. The uniform bound on the number of
cookies per site will be denoted by $M\ge 1$, $M\in\N$.  Informally
speaking, a cookie environment is constructed by placing a pile of
cookies at each site of the lattice (see Figure~\ref{informal}). 
The piles of cookies represent the transition probabilities of the
random walker: upon each visit to a site the walker consumes the
topmost cookie from the pile at that site and makes a unit step to the
right or to the left with probabilities prescribed by that cookie.
If the cookie pile at the current site is empty the walker makes a
unit step to the right or to the left with equal probabilities.
\begin{figure}[h]
  \centering
  \epsfig{file=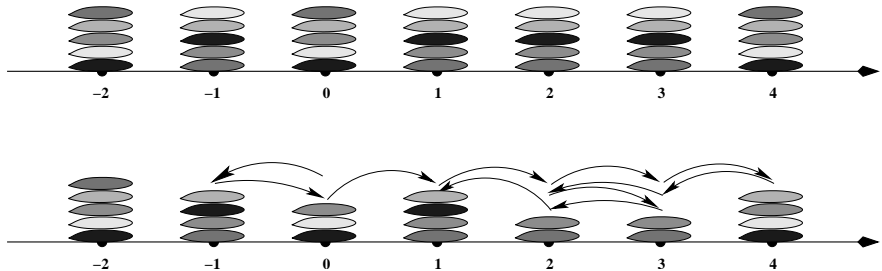,height=3.5cm, angle=0}
\caption{\footnotesize The top picture is an example of an i.i.d.\ cookie
environment with $M=5$, which consists of two types of cookie piles.
An independent toss of a fair coin determines which type of cookie
pile is placed at each site of the lattice. Various shades of gray
allude to different transition probabilities associated to different
cookies. The bottom picture shows the first few possible steps of a
random walker in this cookie environment starting at 0.}
\label{informal}
\end{figure}  

A cookie will be called positive (resp.\ negative) if its consumption
makes the walker to go to the right with probability larger (resp.\
smaller) than 1/2.  A cookie which is neither positive nor negative
will be called a placebo.  Placebo cookies allow us to assume without
loss of generality that each pile originally consists of exactly $M$
cookies.  Unless stated otherwise, the random walker always starts at
the origin.

The term ``excited random walk" was introduced by Benjamini and Wilson
in \cite{bewi}, where they considered random walks on $\Z^d$, $d\ge
1$, in an environment of identical cookies, one per each site.
Allowing more (or fewer) than one cookie per site and randomizing the
environment naturally gave rise to the multi-excited random walk model
in random cookie environments. We refer to \cite{Ze05} and \cite{Ze06} 
for the precise description
and first results.
It was clear then that this new model exhibits a very interesting behavior
 for $d=1$.  We shall mention some of the results for $d\ge 2$
in Section~\ref{mult} and now concentrate on the one-dimensional case.

The studies of excited random walks on integers were continued in \cite{MPV06},
\cite{BS06}, and \cite{BS07}.  \cite{AR05} deals with numerical
simulations of this model.  In all papers mentioned above a (possible)
bias introduced by the consumption of a cookie was assumed to be only
in one direction, say, positive.  The recurrence and transience,
strong law of large numbers \cite{Ze05}, conditions for positive
linear speed \cite{MPV06}, \cite{BS06}, and the rates of escape to
infinity for transient walks with zero speed \cite{BS07} are now well
understood.  Yet some of the methods and facts used in the proofs (for
example, comparison with simple symmetric random walks, submartingale
property) depend significantly on this ``positive bias'' assumption.
The main novelty of the current paper is in considering cookie
environments, which may induce positive or negative drifts at
different sites or even at the same site on successive visits. Our
main results are the recurrence/transience criterion (Theorem 1),
the criterion for positive linear speed (Theorem 2) 
and an annealed central limit theorem (Theorem 3). The 
first two theorems are extensions of those for non-negative cookie
environments but we believe that this is a purely one-dimensional
phenomenon. Moreover, in Section~\ref{mult} we give an example, which
shows that, at least for $d\ge 4$, the criteria for recurrence or
transience and for positive linear speed can not depend just on a
single parameter, the average total drift per site (see (\ref{D})).
The order of the cookies in the pile should matter as well.

The proofs are based on the connections to branching processes with
migration. Branching processes allowing both immigration and emigration
were studied by several authors in the late 70-ties through about the
middle of the 90-ties, and we use some of the results from the
literature (Section~\ref{lit}). See the review paper \cite{VZ93} for more
results and an extensive list of references up to about 1990.  
The connection between one-dimensional random walks and branching
processes was observed long time ago.  In particular, it was used for
the study of random walks in random environments, see e.g.\
\cite{KKS75}. In the context of excited random walks, this idea was employed
recently in \cite{BS06}, \cite{BS07} (still under the ``positive
bias'' assumption).  In the present paper we are using results from
the literature about branching processes with migration in a more
essential way than \cite{BS06} and \cite{BS07}.  One of our tasks is
to show how to translate statements about excited random walks  into statements for a
class of branching processes with migration which have been studied in
the past.

Let us now describe our model, which we shall abbreviate by ERW, more precisely. 
A cookie environment $\om$ with $M$ cookies per site
$z\in\Z$ is an element of
\begin{align*}
  \Omega_M:=\big\{((\om(z,i))_{i\in\N})_{z\in\Z}
 \mid &\ \om(z,i)\in[0,1],
  \forall i\in\{1,2,\dots,M\}\\
  &\text{\ and}\  
  \om(z,i)=1/2, \forall i>M,\ \forall z\in\Z\big\}.
\end{align*}
The purpose of $\om(z,i)$ is to serve as the transition probability from
$z$ to $z+1$ of a nearest-neighbor ERW upon the $i$-th
visit to a site $z$.  More precisely, for fixed $\omega\in \Omega_M$
and $x\in \Z$ an ERW $(X_n)_{n\ge 0}$ starting from $x$
in the cookie environment $\omega$ is a process on a suitable
probability space with probability measure $P_{x,\om}$ which
satisfies:
\begin{align*}
  P_{x,\omega}[X_0=x]&=1,\\P_{x,\omega}[X_{n+1}=X_n+1\,|(X_i)_{0\le i\le n}]&=
  \omega(X_n,\#\{i\le n\,|X_i=X_n\}),\\ 
  P_{x,\omega}[X_{n+1}=X_n-1\,|(X_i)_{0\le i\le n}]&=
  1-\omega(X_n,\#\{i\le n\,|X_i=X_n\}).
\end{align*}
The cookie environment $\om$ may be chosen at random itself according
to a probability measure on $\Om_M$, which we shall denote by $\PP$,
with the corresponding expectation operator $\EE$.  Unless stated
otherwise, we shall make the following assumption on $\PP$:
\begin{equation}\label{H0}
\mbox{The sequence $(\om(z,\cdot))_{z\in\Z}$ is i.i.d.\ under $\PP$.}
\end{equation}
Note that assumption (\ref{H0}) does not imply independence between
different cookies at the same site but only between cookies at
different sites, see also Figure \ref{informal}.  To avoid degenerate
cases we shall also make the following mild ellipticity assumption on
$\PP$:
\begin{equation}
\label{elli} 
  \EE\left[\prod_{i=1}^M
  \om(0,i)\right]>0\ \text{ and }\ \EE\left[\prod_{i=1}^M
  (1-\om(0,i))\right]>0.
\end{equation}
After consumption of a cookie $\om(z,i)$ the random walk is displaced
on $P_{x,\om}$-average by $2\om(z,i)-1$. This average displacement, or
drift, is positive for positive cookies and negative for negative
ones. The consumption of a placebo cookie results in a symmetric
random walk step. Averaging the drift over the environment and summing
up over all cookies at one site defines the parameter
\begin{equation}
  \label{D}
  \delta\ :=\ \EE\left[\sum_{i\ge1}(2\om(0,i)-1)\right]\ =\ \EE\left[\sum_{i=1}^M(2\om(0,i)-1)\right],
\end{equation}
which we shall call the \textit{average total drift per site}. It
plays a key role in the classification of the asymptotic behavior of
the walk as shown by the three main theorems of this paper.

Our first result extends \cite[Theorem 12]{Ze05} about recurrence and
transience for non-negative cookies to i.i.d.\ environments with a
bounded number of positive and negative cookies per site.
\begin{theorem}[Recurrence and transience]\label{rt0} 
  If $\delta \in[-1,1]$ then the walk is recurrent, i.e.\ for
  $\PP$-a.a.\ environments $\om$ it returns $P_{0,\om}$-a.s.\
  infinitely many times to its starting point. If $\delta >1$ then the
  walk is transient to the right, i.e.\ for $\PP$-a.a.\ environments
  $\om$, $X_n\to\infty$ as $n\to\infty$ $P_{0,\om}$-a.s.. Similarly,
  if $\delta <-1$ then the walk is transient to the left, i.e.\
  $X_n\to-\infty$ as $n\to\infty$.
\end{theorem}
Trivial examples with $M=1$ and $\om(0,1)=0$ or $\om(0,1)=1$ show that
assumption (\ref{elli}) is essential for Theorem~\ref{rt0} to hold.

Our next result extends \cite[Theorem 1.1, Theorem 1.3]{MPV06} and
\cite[Theorem 1.1]{BS07} about the positivity of speed from spatially
uniform deterministic environments of non-negative cookies to i.i.d.\
environments with positive and negative cookies.
\begin{theorem}[Law of large numbers and ballisticity]\label{v0} 
  There is a deterministic $v\in[-1,1]$ such
  that the excited random walk 
satisfies  for $\PP$-a.a.\ environments $\om$,
\[\lim_{n\to\infty}
  \frac{X_n}{n}=v\quad P_{0,\om}\text{-a.s..}\] 
Moreover,
  $v<0$ for $\delta <-2$, $v=0$ for
  $\delta\in[-2,2]$ and $v>0$ for $\delta >2$.
\end{theorem}
While Theorems \ref{rt0} and \ref{v0} give necessary and sufficient
conditions for recurrence, transience, and the positivity of the speed,
the following central limit theorem gives only a sufficient condition.
To state it we need to introduce the \textit{annealed}, or
\textit{averaged}, measure $P_x[\ \cdot\ ]:=\EE\left[P_{x,\omega}[\
  \cdot\ ]\right]$.
\begin{theorem}[Annealed central limit theorem]\label{clt}
  Assume that $|\delta|>4$. Let $v$ be the velocity given by Theorem
  \ref{v0} and define
  \[B_t^n:=\frac{1}{\sqrt{n}}(X_{\lfloor tn\rfloor}-\lfloor tn\rfloor
  v)\qquad \mbox{for $t\ge 0$.}
\]
Then $(B_t^n)_{t\ge 0}$ converges in law under $P_0$ to a
non-degenerate Brownian motion with respect to the Skorohod topology
on the space of cadlag functions. 
\end{theorem}
The variance of the Brownian motion in Theorem \ref{clt} will be
further characterized in Section \ref{rs}, see (\ref{si2}). 

Let us describe how the present article is organized. Section~\ref{lit}
introduces the main tool for the proofs, branching processes with
migration, and quotes the relevant results from the literature. In
Sections~\ref{sunny} and \ref{brpr} we describe the relationship between ERW and branching
processes with migration and introduce the necessary notation.  In Section~\ref{rectra} we use this relationship to
translate results from Section~\ref{lit} about branching processes into
results for ERW concerning recurrence and transience, thus proving
Theorem~\ref{rt0}.  In Section \ref{rs} we introduce a renewal structure for
ERW, similar to the one which appears in the study of random walks in
random environments (RWRE), and relate it to branching processes with
migration. In Sections~\ref{Slln} and \ref{sclt} we use this renewal structure to deduce
Theorems~\ref{v0} and \ref{clt}, respectively, from results stated in
Section~\ref{lit}.  The final section contains some concluding remarks and
open questions.

Throughout the paper we shall denote various 
constants by $c_i\in(0,\infty)$, $i\ge 1$.  
\section{Branching processes with migration -- results from the
  literature}\label{lit}
In this section we define a class of branching processes with
migration and quote several results from the literature. We chose to
give the precise statements of the results that we need since some of
the relevant papers are not readily available in English.  
\begin{defi}\label{defi} {\rm  Let $\mu$ and $\nu$ be probability measures on
    $\N_0:=\N\cup\{0\}$ and $\Z$, respectively,
 and let $\xi_{i}^{(j)}$ and $\eta_k$\ $(i,j\ge 1,\ k\ge 0)$ be 
    independent random variables such that each $\xi_{i}^{(j)}$ has
    distribution $\mu$ and each $\eta_k$ has distribution $\nu$. Then
    the process $(Z_k)_{k\ge 0}$, recursively defined by
\begin{equation}\label{Z}
Z_0:=0,\qquad Z_{k+1}:=\xi^{(k+1)}_1+\ldots+\xi^{(k+1)}_{Z_k+\eta_k},\quad k\ge 0,
\end{equation}
is said to be a \textit{$(\mu,\nu)$-branching process} with
\textit{offspring distribution} $\mu$ and \textit{migration
  distribution} $\nu$.  (Here we make an agreement that
$\xi_1^{(k+1)}+\ldots+\xi_i^{(k+1)}=0$ if $i\le 0$.) An offspring
distribution $\mu$ which we shall use frequently is the geometric distribution
with parameter 1/2 and support $\N_0$. It is denoted by \geo.}
\end{defi}
Note that any $(\mu,\nu)$-branching process is a time homogeneous
Markov chain, whose distribution is determined by $\mu$ and $\nu$.
More precisely, if at time $k$ the size of the population is $Z_k$
then (1) $\eta_k$ individuals immigrate or $\min\{Z_k,|\eta_k|\}$
individuals emigrate depending on whether $\eta_k\ge 0$ or $\eta_k<0$
respectively, and (2) the resultant $(Z_k+\eta_k)_+$ individuals
reproduce independently according to the distribution $\mu$.  This
determines the size $Z_{k+1}$ of the population at time $k+1$.

In the current paper we are interested in the case when both the
immigration and the emigration components are non-trivial and the
number of emigrants is bounded from above.  This bound will be the
same as the bound $M$ on the number of cookies per site. 
We shall assume that
\begin{equation}\label{-M}
\nu(\N)>0\quad\text{and}\quad \nu(\{k\in\Z\mid k\ge -M\})=1.
\end{equation}
Denote the average migration by
\begin{equation}\label{defla}
\lambda:=\sum_{k\ge -M} k\, \nu(\{k\})
\end{equation} 
and the moment generating function of the offspring distribution by
\[f(s):=\sum_{k\ge 0}\mu(\{k\})\, s^k,\quad s\in[0,1].\] In addition
to (\ref{-M}), we shall make the following assumptions on the measures
$\mu$ and $\nu$:
\begin{itemize}
\item [(A)]\qquad $f(0)>0$,\quad $f'(1)=1$,\quad
  $b:=f''(1)/2<\infty$,\quad $\lambda<\infty$;
\item [(B)]\qquad ${\displaystyle \sum_{k\ge 1} \mu(\{k\})\,k^2\,
    \ln k <\infty}$.
\end{itemize}
Note that $\mu=\geo$ satisfies condition (A) on the moment generating
function $f$ with $b=1$. It also satisfies (B). 

Next we state a result from the literature, which relates the limiting
behavior of the process $(Z_k)_{k\ge 0}$ to the value of the parameter
\begin{equation}\label{theta}
\theta:=\frac{\lambda}{b}.
\end{equation}
At first, introduce the stopped process $(\wZ_k)_{k\ge 0}$. Let
\begin{equation}
 N^{(Z)}:=\inf\{k\ge 1\mid Z_{k}=0\}\quad 
  \mbox{and}\quad
\label{wZ}
\wZ_k:=Z_{k}\won_{\{k< N^{(Z)}\}}.
\end{equation}
Note that the process $(\wZ_k)_{k\ge 0}$ follows $(Z_k)_{k\ge 0}$ until the first time $(Z_k)_{k\ge 0}$ 
returns to 0. Then $(\wZ_k)_{k\ge 0}$ stays at 0 whereas $(Z_k)_{k\ge 0}$ eventually regenerates due to the presence of immigration (see the first inequality in (\ref{-M})).
\begin{theo}[\cite{FY89}, \cite{FYK90}]\label{A} 
  Let $(Z_k)_{k\ge 0}$ be a $(\mu,\nu)$-branching process satisfying
  {\rm (\ref{-M})}, {\rm (A)} and {\rm (B)}. We let
\[u_n:=P[ N^{(Z)}>n]=P[\tilde{Z}_n>0],\quad n\in\N,\] 
describe the tail of the distribution of 
$ N^{(Z)}$ 
and denote by
\[v_n:=E\bigg[\sum_{m=0}^n \wZ_m\bigg]\] the expectation of the
total progeny of $(\wZ_k)_{k\ge 0}$ up to time $n\in\N_0\cup\{\infty\}$.
   Then the following statements hold.
  \begin{itemize}
  \item [(i)] If $\theta>1$ then
    $\lim\limits_{n\to\infty}u_n=\con{eins}\in(0,1)$, in particular, the process
    $(\wZ_k)_{k\ge 0}$ has a strictly positive chance $c_{\ref{eins}}$ never
    to die out.
  \item [(ii)] If $\theta=1$ then $\lim\limits_{n\to\infty}u_n\ln
    n=\con{zwei}\in(0,\infty)$, in particular, the process $(\wZ_k)_{k\ge 0}$ will
    eventually die out a.s..
  \item [(iii)] If $\theta=-1$ then
    $\lim\limits_{n\to\infty}v_n(\ln n)^{-1}=\con{drei}\in(0,\infty)$, in
    particular, $v_\infty=\infty$, i.e.\ the expected total progeny of
    $(\wZ_k)_{k\ge 0}$, $v_\infty$, is infinite.
  \item [(iv)] If $\theta< -1$ and
    \[\sum_{k\ge 1} k^{1+|\theta|}\mu(\{k\})<\infty\] then
    $\lim\limits_{n\to\infty}u_nn^{1+|\theta|}=\con{four}\in(0,\infty)$.
    Moreover, in this case
    $\lim\limits_{n\to\infty}v_n=\con{five}\in(0,\infty)$, i.e.\ the expected
    total progeny of $(\wZ_k)_{k\ge 0}$ is finite.
  \end{itemize}
\end{theo}
The above results about the limiting behavior of $u_n$ are contained
in Theorems 1 and 4 of \cite{FY89}, \cite{FYK90}. The proofs are given
only in \cite{FYK90}. The behavior of $v_n$ is the content of formula
(33) in \cite{FYK90}.  The statements (i) and (ii) of Theorem~\ref{A} also follow from
\cite[Theorem 2.2]{YY95} (see also \cite[Theorem 2.1]{YMY03}).
\begin{rem} {\rm We have to point out that we use a slightly different
    (and more convenient for our purposes) definition of the lifetime,
    $N^{(Z)}$, of the stopped process. More precisely, our quantity
    $u_n$ can be obtained from the one in \cite{FYK90} by the shift of
    the index from $n$ to $n-1$ and multiplication
    by \[P[\wZ_1>0]=\sum_{k\ge 1} \nu(\{k\})\left(1-\mu(\{0\})^k\right),\]
    which is positive due to the first inequality in (\ref{-M}) and
    the fact that $\mu(\{0\})<1$ (by the condition $f'(1)=1$ of
    assumption (A)). A similar change is needed for the expected total
    progeny of the stopped process. Clearly, these modifications
    affect only the values of constants in Theorem~\ref{A} and not
    their positivity or finiteness. }
\end{rem}

The papers mentioned above contain other results but we chose to state
only those that we need.  In fact, we only use the first part of (iv) and  the
following characterization, which we obtain from Theorem~\ref{A} by
a coupling argument.
\begin{cor}\label{AA}
  Let the assumptions of Theorem \ref{A} hold. Then $(\wZ_k)_{k\ge 0}$ dies
  out a.s.\ iff $\theta\le 1$.  Moreover,
  the expected total progeny of $(\wZ_k)_{k\ge 0}$, $v_\infty$, is finite iff
  $\theta<-1$.
\end{cor}
\begin{proof}
  Theorem \ref{A} (i) gives the `only if'-part of the first statement.
  To show the 'if' direction we 
assume that $\theta\le 1$, i.e.\  $\nu$ has mean $\la\le b$ (see
  (\ref{theta})).  Then there is another $\nu'$ with mean $b$ which
  stochastically dominates $\nu$. Indeed, if $X$ has distribution
  $\nu$ and $Y$ has expectation $b-\la$ and takes values in $\N_0$
  then $\nu'$ can be chosen as the distribution of $X+Y$.  By
  coupling, the $(\mu,\nu')$-branching process stochastically
  dominates the $(\mu,\nu)$-branching process. However, the
  $(\mu,\nu')$-branching process dies out a.s.\ due to Theorem \ref{A}
  (ii) since for this process $\theta=1$.  Consequently, the
  $(\mu,\nu)$-branching process must die out, too.

  Similarly, Theorem \ref{A} (iv) gives the `if'-part of the second statement.
  The converse direction follows from monotonicity as
  above and Theorem \ref{A} (iii).
\end{proof}
\section{From ERWs to branching processes with migration}\label{sunny}

The goal of this mainly expository section is to show how our ERW
model can be naturally recast as a branching process with
migration. This connection was already observed and used in
\cite{BS06} and \cite{BS07}.

Consider a nearest neighbor random walk path $(X_n)_{n\ge 0}$, which
starts at $0$ and define
\[T_k:=\inf\{n\ge 1\mid X_n=k\}\in \N\cup\{\infty\},\quad k\in\Z.\]
Assume for the moment that $X_1=1$ and consider the right excursion,
i.e.  $(X_n)_{0\le n<T_0}$. The left excursion can be treated by
symmetry.

On the set $\{T_0<\infty\}$ we can define a bijective path-wise mapping of this
right excursion to a finite rooted tree, which corresponds to a
realization of a branching process with the extinction time
$\max\{X_n,\ 0\le n<T_0\}$ as illustrated in Figure~\ref{tree_leaves}.
\begin{figure}[t]
\begin{center}
\epsfig{file=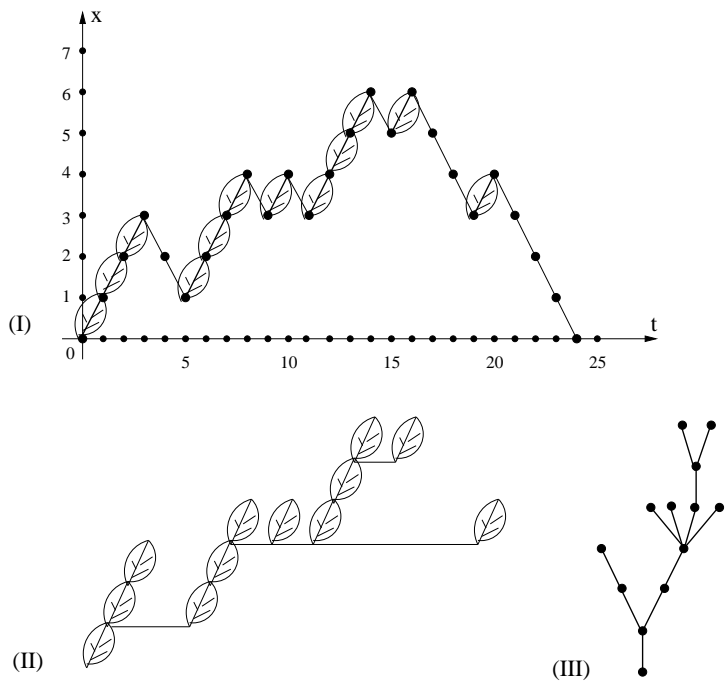,height=9cm, angle=0}
\caption{\footnotesize (I) Right excursion of the random walk.
  Upcrossings are marked by ``tree leaves''. (II) The number of
  upcrossings of the edge $(k,k+1)$ becomes the number of particles
  in generation $k$ for the branching process. Shrinking the horizontal
lines in (II) into single points gives the tree (III). Traversing the
  tree (III) in preorder rebuilds the
  excursion (I).}
\label{tree_leaves}
\end{center}
\end{figure}
Moreover, given a tree for a branching process that becomes extinct in
finite time, we can reconstruct the right excursion of the random
walk. This can be done by making a time diagram of up and down
movements of an ant traversing the tree in preorder: the ant starts at
the root, always chooses to go up and to the left whenever possible,
never returns to an edge that was already crossed in both directions,
and finishes the journey at the root (Figure~\ref{tree_leaves},
(III)).

The above path-wise correspondence on $\{T_0<\infty\}$ does not depend
on the measure associated to the random walk paths. 
To consider the set $\{T_0=\infty\}$ we shall need some of the
properties of this measure. The following simple statement leaves only
three major possibilities for a long term behavior of an ERW path.
\begin{lemma}\label{neww}
Let $\om\in\Om_M$. Then $P_{0,\om}$-a.s.\ 
\[\liminf_{n\to\infty} X_n,\ \limsup_{n\to\infty}X_n\in\{-\infty,+\infty\}.\]
\end{lemma}
\begin{proof}
Let $z\in\Z$. If the ERW visits $z$ infinitely many times then it also
visits $z+1$ infinitely many times due to the second Borel Cantelli
lemma, the strong Markov property, and the assumption $\om(z,i)=1/2$
for $i>M$.  This implies $P_{0,\om}$-a.s.\ $\limsup_nX_n\notin
\Z$. Similarly, $P_{0,\om}$-a.s.\ $\liminf_nX_n\notin \Z$.
\end{proof}

Let us now put a measure on the paths and see what kind of measure will be induced on trees. Consider the right excursion of the simple symmetric random walk. Assume without loss of generality that the walk starts at 1. Then the probability that $T_0<\infty$ is equal to one and the corresponding measure on trees will be the one for a standard Galton-Watson
process with the \geo offspring distribution starting from a single
particle. More precisely, set $U_0:=1$ and let
\begin{equation}
  \label{U}
  U_k:=\#\{n\ge 0\mid n<T_0,\ X_n=k,\ X_{n+1}=k+1\},\qquad k\ge 1,
\end{equation}
be the number of upcrossings of the edge $(k,k+1)$ by the walk before
it hits 0. Then $(U_k)_{k\ge 0}$ has the same distribution as the Galton-Watson process with \geo offspring distribution. Therefore, $(U_k)_{k\ge 0}$ can also be generated as follows: 
start with one particle: $U_0=1$. To generate the $(k+1)$-st
generation from the $k$-th generation (assuming that the process has not yet died out), 
the first particle of generation $k$ tosses a fair coin repeatedly 
and produces one offspring if the coin comes up "heads". It stops the reproduction once the coin comes up "tails". Then the second particle in generation $k$ follows the same procedure independently, then the third one, and so on. Consequently,
\[U_{k+1}=\xi^{(k+1)}_1+\ldots+\xi^{(k+1)}_{U_k},\]
where $\xi_i^{(j)}$, $i,j\ge 1$ are independent with distribution
\geo.

To construct a branching process corresponding to an ERW with $M$ cookies per site one can use exactly the 
same procedure except that for the first $M$ coin tosses in the $k$-th generation the particles should use coins with biases "prescribed" by the cookies located at site $k$. Since every particle tosses a coin at least
once, at most the first $M$ particles in each generation will have a chance to use biased coins. All the remaining particles will toss fair coins only.  This can be viewed as a branching process with migration in the following natural way. Before the reproduction starts, the first $U_k\wedge M$ particles emigrate, taking with them all $M$ biased coins and an infinite supply of fair coins. In exile
they reproduce according to the procedure described above. Denote the total number of offspring produced by these particles by $\eta^{(k+1)}_{U_k\wedge M}$.  Meanwhile, the
remaining particles (if any) reproduce using only fair coins. Finally, the offspring of the emigrants re-immigrate.  Therefore, the number of particles in the generation $k+1$ can be written as
\begin{equation}\label{ejp}
U_{k+1}:=\xi^{(k+1)}_1+\ldots+\xi^{(k+1)}_{U_k-M}+\eta^{(k+1)}_{U_k\wedge
M},
\end{equation}
where $\xi_i^{(j)}$ and $\eta^{(k)}_{\ell}$ $(i,j,k\ge 1,\ 
0\le \ell\le M)$ are independent random variables, each one of the sequences
$(\eta^{(k)}_{0})_{k\ge 1},\ldots, (\eta^{(k)}_{M})_{k\ge 1}$ is identically distributed, and each $\xi_i^{(j)}$ has distribution \geo. 

 Branching processes of type (\ref{ejp}) were considered in \cite{BS06} (p.\ 630) and \cite{BS07} (p.\
815), except that they were generated not by the forward but by the backward excursion (see (\ref{Uk}) in Section~\ref{rs}). Careful analysis of such processes carried out in these two papers
yielded results concerning positive speed and rates of growth at infinity for ERWs with non-negative cookies. However, from a practical point of view, $(\mu,\nu)$-branching processes, introduced in Definition \ref{defi}, seem to be well-known and studied more extensively in the past. In particular, we could find the results we need (see Theorem \ref{A}) in the literature only for $(\mu,\nu)$-branching processes, but not for processes of the form (\ref{ejp}).  For this reason one of our main tasks will be to relate these two classes of processes in order to translate the results from the literature into results about processes of the form (\ref{ejp}).

\section{Coin-toss construction of the ERW and the related 
$(\mu,\nu)$-branching process}\label{brpr}

In this section we formalize a coin-toss construction of the ERW 
and introduce auxiliary processes used
in the rest of the paper.

Let $(\Om,\F)$ be some measurable space equipped with a family of
probability measures $P_{x,\om},\ x\in\Z,\ \om\in\Om_M$, such that for
each choice of $x\in\Z$ and $\om\in\Om_M$ we have $\pm 1$-valued
random variables $Y_{i}^{(k)},\ k\in\Z,\ i\ge 1,$ which are
independent under $P_{x,\om}$ with distribution given by
\[P_{x,\om}[Y_{i}^{(k)}=1]=\om(k,i)\quad\mbox{ and }\quad
P_{x,\om}[Y_{i}^{(k)}=-1]=1-\om(k,i).\] Moreover, we require that
there is a random variable $X_0$ on $(\Om,\F,P_{x,\om})$ such that
$P_{x,\om}[X_0=x]=1$.  Then an ERW $(X_n)_{n\ge 0}$, starting at
$x\in\Z$, in the environment $\omega$ can be realized on the
probability space $(\Om,\F,P_{x,\om})$ recursively by:
\begin{eqnarray}
X_{n+1}&:=&X_n+Y_{\#\{i\le n\mid X_i=X_n\}}^{(X_n)},\quad n\ge 0.\label{x1}
\end{eqnarray}
We shall refer to $\{Y_{i}^{(k)}=1\}$ as a ``success''
and to $\{Y_{i}^{(k)}=-1\}$ as a ``failure''. 
Due to (\ref{x1}) every step to the right or to the left of the random
walk corresponds to a success or a failure, respectively. 

We now  describe various branching processes that appear in the proofs.  
Namely, we introduce processes $(V_k)_{k \ge 0}$, $(W_k)_{k
  \ge 0}$, and $(Z_k)_{k \ge 0}$. Modifications of the first two
processes suitable for left excursions will be defined later when they
are needed (we shall keep the same notation though, hoping that this
will not lead to confusion). The last process, $(Z_k)_{k \ge 0}$,
will belong to the class of processes from Section~\ref{lit}.

For $m\in\N$ and $k\in \Z$ let
\begin{equation}\label{sucks}
S_0^{(k)}:=0,\quad S_m^{(k)}
  := \#\text{ of successes in
$\big(Y_{i}^{(k)}\big)_{i\ge 1}$ prior to the $m$-th failure.}
\end{equation}
Recall from the introduction that $P_x[\,\cdot\,]$ denotes the averaged
measure $\EE[P_{x,\om}[\,\cdot\,]]$. By assumption (\ref{elli}) the walk
reaches $1$ in one step with positive $P_0$-probability.  

We shall be interested in the behavior of the process $(U_k)_{k\ge 0}$ defined in (\ref{U}). At first, we shall relate $(U_k)_{k\ge 0}$ to 
$(V_k)_{k\ge 0}$ which is  recursively defined by
\begin{equation}\label{def}
V_0:=1,\quad 
V_{k+1}:=S_{V_k}^{(k+1)},\quad k\ge 0.
\end{equation}
Observe that $(V_k)_{k\ge 0}$ is a time homogeneous Markov chain, as
the sequence of sequences $(S_m^{(k)})_{m\ge 0}$, $k\ge 0$, is
i.i.d.. Moreover, $0$ is an absorbing state for $(V_k)_{k\ge 0}$. We
claim that under $P_1$,
\begin{align}
  U_k&=V_k\quad \text{for all $k\ge 0$ on the event
    $\{T_0<\infty\}$};\label{rec}\\
  U_k&\le V_k\quad \text{for all $k\ge 0$ on the event
    $\{T_0=\infty\}$}.\label{nrec}
\end{align}
The relation (\ref{rec}) is obvious from the discussion in Section \ref{sunny} and Figure~\ref{tree_leaves}. To show
(\ref{nrec}) we shall use induction. Recall that $U_0=V_0=1$ and
assume $U_i\le V_i$ for all $i\le k$. From Lemma \ref{neww} we know 
that $X_n\to\infty$ as $n\to\infty$ on $\{T_0=\infty\}$ a.s.\ with
respect to $P_1$.  Therefore, the last, $U_k$-th, upcrossing of the
edge $(k,k+1)$ by the walk is not matched by a downcrossing. This
implies that $U_{k+1}$ should be less than or equal to the number of
successes in the sequence $\big(Y^{(k+1)}_i\big)_{i\ge 1}$ prior to
the $U_k$-th failure.  On the other hand, to get the value of
$V_{k+1}$ one needs to count all successes in this sequence until the
$V_k$-th failure.  Since $U_k\le V_k$, we conclude that $U_{k+1}\le
V_{k+1}$.

Next we introduce the process $(W_k)_{k\ge 0}$ by setting
\begin{equation}
  \label{W}
  W_0:=0,\quad W_{k+1}:=S^{(k)}_{W_k\vee M},\quad k\ge 0. 
\end{equation}
Just as $(V_k)_{k\ge 0}$, the process $(W_k)_{k\ge 0}$ is a time
homogeneous Markov chain on non-negative integers. Moreover, the
transition probabilities from $i$ to $j$ of these two processes
coincide except for $i\in\{0,1,\dots, M-1\}$ and both processes can reach any positive number with positive probability. Therefore, if one of
these two processes goes to infinity with positive probability, so
does the other:
\begin{equation}
  \label{vw}
  P_{1}[V_k\to\infty]>0\quad \Longleftrightarrow \quad P_{1}[W_k\to\infty]>0.
\end{equation}
Finally, we decompose the process $(W_k)_{k\ge 0}$ into two components
as follows.  
\begin{lemma}
  \label{deco1}
  For $k\ge 0$ let $Z_k:=W_{k+1}-S^{(k)}_M$.  Then $(Z_k)_{k\ge 0}$ is
  a $(\geo,\nu)$-branching process, where $\nu$ is the common
  distribution of $\eta_k:=S_M^{(k)}-M$ under $P_1$.
\end{lemma}
\begin{proof}
By definition, $Z_0=0$ and
\begin{equation*}
  Z_{k+1}=W_{k+2}-S_M^{(k+1)}\ 
\stackrel{(\ref{W})}{=}\ S_{W_{k+1}\vee M}^{(k+1)}-S_M^{(k+1)}
=\xi^{(k+1)}_1+\dots+\xi^{(k+1)}_{W_{k+1}-M},
\end{equation*}
where $\xi^{(k+1)}_i$ is defined as the number of successes in
$\big(Y_j^{(k+1)}\big)_{j\ge 1}$ between the $(M+i-1)$-th and the
$(M+i)$-th failure, $i\ge 1$. Therefore, by definition of $Z_k$ and
$\eta_k$,
 \[ Z_{k+1}=
 \xi^{(k+1)}_1+\dots+\xi^{(k+1)}_{Z_k+S^{(k)}_M-M}
   =\xi^{(k+1)}_1+\dots+\xi^{(k+1)}_{Z_k+\eta_k}.
\]
Since $\om(k,m)=1/2$ for $m>M$, the random variables $Y_m^{(k)}$, $m>M$,
$k\ge 0$, are independent and uniformly distributed on $\{-1,1\}$
under $P_1$. From this we conclude that the $\xi^{(k)}_i$, $k,i\ge 1,$
have distribution \geo.  To show the independence of $\xi_i^{(j)}$ and
$\eta_k$\ $(i,j\ge 1,\ k\ge 0)$, as required by Definition \ref{defi},
notice that $\eta_k=S_M^{(k)}-M$ depends only on $Y^{(k)}_m$, where
$m$ changes from $1$ to the number of the trial resulting in the
$M$-th failure inclusively, while each $\xi^{(k)}_i$, $i\ge 1$, counts
the number of successes in $\big(Y_m^{(k)}\big)_{m\ge 1}$ between the
$(M+i-1)$-th and the $(M+i)$-th failure. Recalling again that
$Y_m^{(k)}$, $m\ge 1$, $k\ge 0$, are independent under $P_1$, we get the
desired independence.
\end{proof}
Having introduced all necessary processes we can now turn to the
proofs of  our results.

\section{Recurrence and transience} \label{rectra}
\begin{defi}\label{rrtr}{\rm The ERW is called \textit{recurrent from
      the right} if the first excursion to the right of 0, if there is
    any, is $P_0$-a.s.\ finite. 
    Recurrence from the left is defined
    analogously.}
\end{defi}
In the next lemma we shall characterize ERW which are recurrent from
the right in terms of branching processes with migration.  At first,
we shall introduce a relaxation of condition (\ref{H0}), which is
needed for the proof of Theorem \ref{rt0}:
\begin{equation}\label{KK}
 \mbox{The sequence $(\om(k,\cdot))_{k\ge K}$ is i.i.d.\ under $\PP$
  for some $K\in\N$.}
\end{equation}
Under this assumption the sequence indexed by $k\ge K$ of sequences
$(Y_{i}^{(k)})_{i\ge 1}$ is i.i.d.\ with respect to $P_0$. In
particular, the sequence $(S_M^{(k)})_{k\ge K}$ is i.i.d.\ under
$P_0$. 
\begin{lemma}\label{tree2} 
  Replace assumption (\ref{H0}) by (\ref{KK}) and assumption
  (\ref{elli}) by 
 \begin{equation}
\label{elK} 
  \EE\left[\prod_{i=1}^M (1-\om(K,i))\right]>0.
\end{equation}
Denote the common distribution of $\eta_k\ :=\ S^{(k)}_{M}-M$,\ $k\ge
K$, under $P_0$ by $\nu$.  Then the ERW is recurrent from the right if
and only if the $(\geo,\nu)$-branching process dies out a.s., i.e.\
reaches state $0$ at some time $k\ge 1$.
\end{lemma}
\begin{proof}
Since we are interested
  in the first excursion to the right we may assume without
  loss of generality that the random walk starts at $1$.
Then, recalling  definition (\ref{U}), we have
$\{T_0=\infty\}\stackrel{P_1}{=}\{\forall k\ge 1\  U_k>0\},$
where $A\stackrel{P_{1}}{=}B$ means that the two events $A$ and $B$
may differ by a $P_{1}$-null-set only.  Indeed, since $U_k$ counts
only upcrossings of the edge $(k,k+1)$ prior to $T_0$, the inclusion
$\supseteq$ is trivial. The reverse relation follows from
Lemma \ref{neww}. 
This together with (\ref{rec}) and (\ref{nrec}) implies that
\begin{equation}\label{lars}
\{T_0=\infty\}\stackrel{P_1}{=}\{\forall k\ge 0\ V_k>0\}. 
\end{equation}
As above $(V_k)_{k\ge K}$ is a time homogeneous Markov chain since
the sequence of sequences $(S_m^{(k)})_{m\ge 0}$,\ $k\ge K$, is
i.i.d..  For any $m$ the transition probability of this Markov chain
from $m\in\N$ to 0 is equal to
\[P_1[S_m^{(K)}=0]= \EE\left[\prod_{i=1}^m (1-\om(K,i))\right],\] which
is strictly positive by (\ref{elK}).  Since 0 is absorbing for
$(V_k)_{k\ge 0}$  we get that $\{\forall k\ge 0\
V_k>0\}\stackrel{P_1}{=}\{V_k\to\infty\}$.  Consequently, by
(\ref{lars}),
$\{T_0=\infty\}\stackrel{P_{1}}{=}\left\{V_k\to \infty\right\}$.
Next we turn to the process $(W_k)_{k\ge 0}$ and recall relation (\ref{vw}).
Thus, 
\begin{equation}
  \label{tw}
  P_{1}[T_0=\infty]=0\quad \Longleftrightarrow \quad P_{1}[W_k\to\infty]=0.
\end{equation}
Finally, we decompose the process $(W_k)_{k\ge 0}$ as in
Lemma~\ref{deco1} by writing $W_{k+1}=Z_k+S^{(k)}_M$ for $k\ge 0$,
where $(Z_k)_{k\ge K}$ is a Markov chain with the transition kernel of
a $(\geo,\nu)$-branching process.
Since the sequence $(S^{(k)}_M)_{k\ge K}$ is i.i.d., this implies that
$\{W_k\to\infty\}\stackrel{P_{1}}{=}\{Z_k\to\infty\}$.  Together with
(\ref{tw}) this shows that the ERW is recurrent from the right iff
$P_0[Z_k\to\infty]=0$.  Since $(Z_k)_{k\ge K}$ is an irreducible
Markov chain this is equivalent to $(Z_k)_{k\ge K}$ being recurrent,
which is equivalent to recurrence of the state 0 for
$(\geo,\nu)$-branching processes.
\end{proof}
\begin{lemma}\label{rr}
  Assume again (\ref{H0}) and (\ref{elli}). If the ERW is recurrent
  from the right then all excursions to the right of 0 are $P_0$-a.s.\
  finite.  If the ERW is not recurrent from the right then it will
  make $P_0$-a.s.\ only a finite number of excursions to the right.
  The corresponding statements hold for recurrence from the left.
\end{lemma}
\begin{proof} 
  Let the ERW be recurrent from the right. By Definition~\ref{rrtr}
  the first excursion to the right is a.s.\ finite. By
  Lemma~\ref{tree2} the corresponding $(\geo,\nu)$-branching process dies out a.s..
  Let $i\ge 1$ and assume that all excursions to the right up to the
  $i$-th one have been proven to be $P_0$-a.s.\ finite. If the ERW
  starts the $(i+1)$-st excursion to the right of 0 then it finds
  itself in an environment which has been modified by the previous $i$
  excursions up to a random level $R\ge 1$, beyond which the
  environment has not been touched yet. Therefore, conditioning on the
  event $\{R=K\}$, $K\ge 1$, puts us within the assumptions of
  Lemma~\ref{tree2}: the random walk starts the right excursion from
  $0$ in a random cookie environment which satisfies (\ref{KK}). But
  the corresponding $(\geo,\nu)$-branching process is still the same
  and, thus, dies out a.s.. Therefore, this excursion, which is the
  $(i+1)$-st excursion of the walk, is a.s. finite on $\{R=K\}$. Since by
  our induction assumption the events $\{R=K\}$, $K\ge 1$, form a
  partition of a set of full measure, we obtain the first statement
  of the lemma.

  For the second statement let 
  \[
D:=\inf\{n\ge 1\mid X_n<X_0\}
\]
be the first time that the walk backtracks
   below its starting point. 
Due to (\ref{elli}), $P_0[X_1=1]>0$. Therefore, since the walk is assumed to be not recurrent from the right,
  \begin{equation}\label{D>0}
  P_0[D=\infty]>0.
  \end{equation}
  Denote by $K_i$ the right-most visited site before the end of the
  $i$-th excursion and define $K_i=\infty$ if there is no $i$-th right
  excursion or if the $i$-th excursion to the right covers $\N$. Then
  the number of $i\ge 1$ such that $K_i<K_{i+1}$, is stochastically
  bounded from above by a geometric distribution with parameter
  $P_0[D=\infty]$. Indeed, each time the walk reaches a level
  $K_i+1<\infty$, which it has never visited before, it has
  probability $P_0[D=\infty]$ never to backtrack again below the level
  $K_i+1$, independently of its past. Therefore, $(K_i)_i$ increases
  only a finite number of times. Hence $P_0$-a.s.\ $R:=\sup\{K_i\mid
  i\ge 1, K_i<\infty\}<\infty$. Now, if the walk did an infinite
  number of excursions to the right, then, $P_0$-a.s.\ $\sup_n X_n=R<\infty$ and $\limsup_n X_n\ge 0$, which is impossible due to Lemma \ref{neww}.
\end{proof}
\begin{prop}\label{rt} The ERW is recurrent from the right if and only 
  if $\delta \le 1$.  Similarly, it is recurrent from the left if and
  only if $\delta \ge -1$.
\end{prop}
For the proof we need the
 next lemma, which relates the parameter $\delta$ of the ERW and the
parameter $\theta$ of the branching process with migration.
\begin{lemma}\label{th}
  Let $\nu$ be the distribution of $S_M^{(0)}-M$ under $P_0$. Then
  $\theta$ defined in (\ref{theta}) for the $(\geo,\nu)$-branching
  process is equal to $\delta$ defined in (\ref{D}).
\end{lemma}
\begin{proof}[Proof of Lemma \ref{th}.] For $\mu=\geo$ the parameter $b$ defined
  in (A) equals 1. Hence, by (\ref{defla}),
  $\theta=\la=E_0[S_M^{(0)}-M]$. Thus it suffices to show that 
\begin{equation}\label{hello}
E_0[S_M^{(0)}]-M=\delta.
\end{equation}
This has already been observed in \cite[Lemma 3.3]{BS07}.
For
completeness, we include a proof.
  Let $F:=\#\{1\le i\le M\mid Y^{(0)}_i=-1\}$ be the number of
  failures among the first $M$ trials.  Then $M-F$ is the number of
  successes among the first $M$ trials.  Therefore, since $S_M^{(0)}$
  is the total number of successes prior to the $M$-th failure,
  $S_M^{(0)}-(M-F)$ is the number of successes after the $M$-th trial
  and before the $M$-th failure. Given $F$, its distribution is negative binomial
 with parameters $M-F$ and $p=1/2$, i.e.\ the $(M-F)$-fold
  convolution of \geo, and therefore has mean $M-F$.  Thus,
  \[
    E_0[S_M^{(0)}-(M-F)]=E_0[E_0[S_M^{(0)}-(M-F)\mid F]]=E_0[M-F].\]
Subtracting $E_0[F]$ from both sides we obtain
\[E_0[S_M^{(0)}] -M= M-2
\sum_{i=1}^M\EE[1-\om(0,i)]=\sum_{i=1}^M(2\EE[\om(0,i)]-1)=\delta.
\]
\end{proof}
\begin{proof}[Proof of Proposition \ref{rt}]
  Due to Lemma \ref{tree2} the walk is recurrent from the right iff
  the $(\geo,\nu)$-branching process dies out a.s., where $\nu$ is the
  distribution of $S_M^{(0)}-M$. By the first statement of Corollary
  \ref{AA} this is the case iff $\theta\le 1$. The first claim of the proposition follows
now from Lemma \ref{th}.
 The second one follows by symmetry.
\end{proof}
\begin{proof}[Proof of Theorem \ref{rt0}]
  If $\delta > 1$ then by Proposition \ref{rt} the walk is not recurrent from the right but
  recurrent from the left.  If the walk returned
  infinitely often to 0 then it would also make an infinite number of
  excursions to the
  right 
  which is impossible due to Lemma \ref{rr}.  Hence the ERW visits 0
  only finitely often. Since any left excursion is finite due to Lemma
  \ref{rr} the last excursion is to the right and is infinite. Consequently, $P_0$-a.s.\ $\liminf_nX_n\geq 0$, and therefore, 
  due to Lemma
  \ref{neww},  $X_n\to \infty$. Similarly, $\delta<-1$ implies $P_0$-a.s.\ $X_n\to\-\infty$.

In the remaining case $\delta \in[-1,1]$ all excursions from 0 are
finite due to Proposition \ref{rt}. Hence, 0 is visited infinitely many times.
\end{proof}
\begin{rem}{\rm The equivalence (\ref{lars}) also holds
    correspondingly for one-dimensional random walks $(X_n)_{n\ge 0}$ in
    i.i.d.\ random environments (RWRE) and branching processes
    $(V_k)_{k\ge 0}$ in random environments, i.e.\ whose offspring distribution is
    geometric with a random parameter. This way the recurrence theorem
    due to Solomon \cite[Th.\ (1.7)]{So75} for RWRE can be
    deduced from results by Athreya and Karlin, see
    \cite[Chapter VI.5, Corollary 1 and Theorem 3]{AN72}.}
\end{rem}
\section{A renewal structure for transient ERW}\label{rs}

A powerful tool for the study of random walks in random environments
(RWRE) is the so-called \textit{renewal} or \textit{regeneration
  structure}. It is already present in \cite{KKS75}, \cite{Ke77} and
was first used for multi-dimensional RWRE in \cite{SZ99}. It has been
mentioned in \cite[p.\ 114, Remark 3]{Ze05} that this renewal
structure can be straightforwardly adapted to the setting of
directionally transient ERW in i.i.d.\ environments in order to give a
law of large numbers.  The proofs of positivity of speed and of a
central limit theorem for once-excited random walks in dimension $d\ge 2$ in \cite{BeR}
were also phrased in terms of this renewal structure.  We shall do the
same for the present model.

We continue to assume (\ref{H0}) and (\ref{elli}). Let $\delta>1$, where $\delta$ is the average drift defined in (\ref{D}). This means, 
due to Theorem \ref{rt0}, that $P_0$-a.s.\ $X_n\to\infty$.
Moreover, by Proposition \ref{rt}, the walk is not recurrent from the right, which implies,
as we already mentioned, see (\ref{D>0}),  that $P_0[D=\infty]>0$.  
Hence there are
$P_0$-a.s.\ infinitely many random times $n$, so-called
\textit{renewal} or \textit{regeneration times}, with the defining property
that $X_m<X_n$ for all $0\le m<n$ and $X_m\ge X_n$ for all $m>n$.
Call the increasing enumeration of these times $(\tau_k)_{k\ge 1}$,
see also Figure~\ref{ff}.
\begin{figure}[t]
\begin{center}
 \psfrag{aa}{$X_{\tau_1}$}
 \psfrag{bb}{$X_{\tau_2}$}
 \psfrag{cc}{$\tau_1$}
 \psfrag{dd}{$\tau_2$}
\psfrag{nn}{$n$}
\epsfig{figure=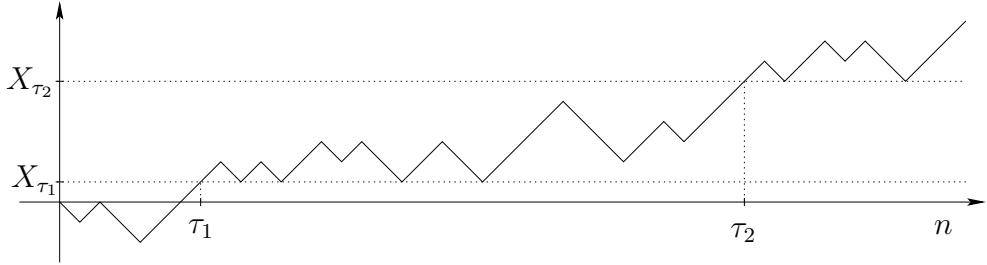,height=100pt}\vspace*{-0mm}
\end{center}
\caption{\footnotesize A random walk path with two renewals.}\label{ff}
\end{figure}
  Then the sequence
$(X_{\tau_1},\tau_1),(X_{\tau_{k+1}}-X_{\tau_k},\tau_{k+1}-\tau_k)\
(k\ge 1)$ of random vectors is independent under $P_0$.
Furthermore, the random vectors
$(X_{\tau_{k+1}}-X_{\tau_k},\tau_{k+1}-\tau_k),\ k\ge 1$, have the
same distribution under $P_0$. For multidimensional RWRE and once-excited random walk the corresponding statement is \cite[Corollary 1.5]{SZ99} and
\cite[Proposition 3]{BeR}, respectively.  It follows from the renewal theorem, see
e.g.\ \cite[Lemma 3.2.5]{zei}, that
\begin{equation}\label{ren}
E_0[X_{\tau_2}-X_{\tau_1}]=P_0[D=\infty]^{-1}<\infty.
\end{equation} Moreover, the ordinary strong law of large numbers implies that 
\begin{equation}\label{v}
  \lim_{n\to\infty}\frac{X_n}{n}=\frac{E_0[X_{\tau_2}-X_{\tau_1}]}
{E_0[\tau_2-\tau_1]}=:v\qquad\mbox{$P_0$-a.s.,}
\end{equation}
see \cite[Proposition 2.1]{SZ99} and \cite[Theorem 3.2.2]{zei} for
RWRE and also \cite[Theorem 2]{BeR} for once-ERW.  Therefore,
\begin{equation}\label{first}
\mbox{$v>0$\quad if and only if \quad  $E_0[\tau_2-\tau_1]<\infty.$}
\end{equation}
  If, moreover,
\begin{equation}\label{second}
E_0[(\tau_2-\tau_1)^2]<\infty
\end{equation}
then the result claimed in Theorem \ref{clt} holds with
\begin{equation}\label{si2}
\si^2:=\frac{E_0\left[\left(X_{\tau_2}-X_{\tau_1}-v(\tau_2-\tau_1)\right)^2\right]}
{E_0[\tau_2-\tau_1]}>0
\end{equation}
see \cite[Theorem 4.1]{Sz00} for RWRE and \cite[Theorem 3 and Remark
1]{BeR} for once-ERW.  

Thus, in order to prove Theorems \ref{v0} and \ref{clt} we need
to control the first and the second moment, respectively, of $\tau_2-\tau_1$. We start by introducing for $k\ge 0$ the number
\begin{equation}\label{Uk}
D_k:=\#\left\{n\mid \tau_1<n<\tau_2,\ X_n=X_{\tau_2}-k,
\ X_{n+1}=X_{\tau_2}-k-1\right\}
\end{equation}
of downcrossings of the edge $(X_{\tau_2}-k,X_{\tau_2}-k-1)$ between
the times $\tau_1$ and $\tau_2$.
\begin{lemma}\label{mom} Assume that the ERW is transient to the right
  and let $p\ge 1$. Then the $p$-th moment of $\tau_2-\tau_1$ under
  $P_0$
  is finite if and only if the $p$-th moment of $\sum_{k\ge 1}D_k$ is
  finite.
\end{lemma}
\begin{proof}
  The number of upcrossings between $\tau_1$ and $\tau_2$ is
  $X_{\tau_2}-X_{\tau_1}+\sum_{k\ge 1}D_k$, since
  $X_{\tau_1}<X_{\tau_2}$ and since each downcrossing needs to be
  balanced by an upcrossing.  Each step is either an
  upcrossing or a downcrossing, therefore,
\begin{equation}\label{pick}
\tau_2-\tau_1=X_{\tau_2}-X_{\tau_1}+2\sum_{k\ge 1}D_k.
\end{equation}
For every $k\in\{X_{\tau_1}+1,\ldots,X_{\tau_2}-1\}$ there is a
downcrossing of the edge $(k,k-1)$, otherwise $k$ would be
another point of renewal. Hence,  $X_{\tau_2}-X_{\tau_1}\le
1+\sum_{k\ge 1}D_k$ and, by (\ref{pick}),
\[2\sum_{k\ge 1}D_k\le \tau_2-\tau_1\le 1+3\sum_{k\ge 1}D_k.\]
This implies the claim.
\end{proof}

To interpret $(D_k)_{k\ge 0}$ as a branching process (see Figure~\ref{ff1}) we define for $m\in\N$ and $k\in \Z$
\begin{equation}\label{fs}
F_0^{(k)}:=0,\quad F_m^{(k)}
  := \#\text{ of failures in
$\big(Y_{i}^{(k)}\big)_{i\ge 1}$ prior to the $m$-th success.}
\end{equation}
(Compare this to the definition of $S_m^{(k)}$ in (\ref{sucks}).)
\begin{figure}[t]
\begin{center}
 \psfrag{xx}{$X_{\tau_1}$}
 \psfrag{yy}{$X_{\tau_2}$}
 \psfrag{a}{$D_0=0$}
\psfrag{b}{$D_1=1$}
\psfrag{c}{$D_2=2$}
\psfrag{d}{$D_3=4$}
\psfrag{e}{$D_4=4$}
\psfig{figure=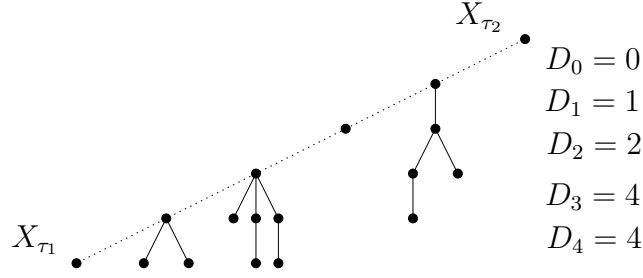,height=100pt}\vspace*{-0mm}
\end{center}
\caption{\footnotesize For the path in Figure~\ref{ff} the process
  $(D_k)_{k\ge 0}$ is realized as (0,1,2,4,4,0,0,\ldots).  The solid
  lines represent downcrossings. The thick dots on the dashed line
  correspond to the single immigrant in definition
  (\ref{VV}).}\label{ff1}
\end{figure}
Let
\begin{align}\label{VV}
V_0&:=0,\quad V_{k+1}\ :=\ F^{(k)}_{V_k+1},\quad k\ge 0;\\
  \label{nv}
  \wV_k&:=V_k\won_{\{k<N^{(V)}\}},\quad\text{ where
}\quad N^{(V)}:=\inf\{k\ge 1\mid V_k=0\}.
\end{align}
\begin{lemma}\label{ffs} Assume that the ERW is transient to the right.
Then $(D_k)_{k\ge 0}$ and $(\wV_k)_{k\ge 0}$ have the same distribution under $P_0$.
\end{lemma}
\begin{proof}
  Fix an integer $K\ge 1$. For brevity, we set $\vec
  D:=(D_1,\ldots,D_K)$ and $\vec V:=(\wV_1,\ldots,\wV_K)$.  It
  suffices to show that
\begin{equation}\label{wet}
P_0\left[\vec D=\vec i\,\right]=P_0\left[\vec V=\vec i\,\right]
\end{equation}
for all $\vec i\in\N_0^K$.  Since both processes start from $0$ and
also stay at $0$ once they have returned to $0$ for the first time, it
is enough to consider vectors $\vec i$ whose entries are strictly
positive except for maybe the last one. And, since the process $(D_k)_{k\ge 0}$
eventually does reach $0$ $P_0$-a.s., namely at
$k=X_{\tau_2}-X_{\tau_1}<\infty,$ it suffices to consider only $\vec
i$ whose last entry is 0.  Thus, let $\vec
i=(i_1,\ldots,i_K)\in\N_0^K$ with $i_1,\ldots,i_{K-1}\ge 1$ and
$i_K=0$.  At first, we shall show that
\begin{equation}\label{french}
P_0\left[\vec D=\vec i\,\right]=P_0\left[\vec D^{(K)}=\vec i\,\right],
\end{equation}
where, for $m,k\ge 0$,
\begin{eqnarray}\label{Uk'}
D_k^{(m)}&:=&\#\{n< T_m\mid X_n=m-k,\ X_{n+1}=m-k-1\}\quad\mbox{and}\\
 \vec D^{(m)}&:=&(D_1^{(m)},\ldots,D_K^{(m)}).\nonumber
\end{eqnarray}
We start from the partition equation
\begin{equation}\label{dipl}
  P_0\left[\vec D=\vec i\,\right]=\sum_{m\ge 0}P_0\left[\vec D=
\vec i,\ \tau_2=T_m\right].
\end{equation}
However, on the event $\{\vec D=\vec i,\ \tau_2=T_m\}$, we have
$X_{\tau_1}=m-K$ by our choice of $\vec i$. Since $X_{\tau_1}\ge 0$, we
may start the summation in (\ref{dipl}) from $m=K$.
Moreover, comparing the definitions (\ref{Uk}) and (\ref{Uk'}), we
see that $\{\vec D=\vec i,\ \tau_2=T_m\}=\{\vec D^{(m)}=\vec i,\
\tau_2=T_m\},$ using that not only on the left but also on the right
event we have $X_{\tau_1}=m-K$.  Hence, the right hand side of
(\ref{dipl}) is equal to
\begin{equation}\label{ewr}
\sum_{m\ge K}P_0\left[\vec D^{(m)}=\vec i,\ \tau_2=T_m\right].
\end{equation}
Notice that the event $\{\tau_2=T_m\}$ occurs if and only if the ERW does
not fall below level $m$ after time $T_m$ and if exactly one of the
numbers $D_1^{(m)},\ldots,D_m^{(m)}$ is equal to 0. Then, by our choice of
$\vec i$, (\ref{ewr}) is equal to
\[\sum_{m\ge K}P_0\left[\vec D^{(m)}=\vec i,\ \forall k=K+1,\ldots,m:
  D_k^{(m)}\ge 1,\ \forall n\ge T_m:\ X_n\ge m\right].
\]
By the strong Markov property applied to the stopping time $T_m$ and
by independence in the environment this equals
\[\sum_{m\ge K}P_0\left[\vec D^{(m)}=\vec i,\ \forall k=K+1,\ldots,m:
  D_k^{(m)}\ge 1\right] P_m[D=\infty].
\]
Since $i_K=0$, this is the same as
\[
\sum_{m\ge K}P_0\left[\vec D^{(m)}=\vec i,\ \forall k=1,\ldots,m-K:
  D_k^{(m-K)}\ge 1\right] P_m[D=\infty].\] Applying the strong Markov
property once more, this time to $T_{m-K}$, and using the i.i.d.\
structure of the environment, we get that the above is equal to
\begin{align*}
\sum_{m\ge K}P_{m-K}&\left[\vec D^{(m)}=\vec i\,\right] 
P_0\left[\forall k=1,\ldots,m-K: D_k^{(m-K)}\ge 1\right]
P_m[D=\infty]\\
=P_{0}&\left[\vec D^{(K)}=\vec i\,\right]\sum_{m\ge K} 
P_0\left[\forall k=1,\ldots,m-K: D_k^{(m-K)}\ge 1\right]
P_{m-K}[D=\infty]\\
=P_{0}&\left[\vec D^{(K)}=\vec i\,\right]\sum_{m\ge K}
P_0[\tau_1=T_{m-K}]\ =\ P_{0}\left[\vec D^{(K)}=\vec i\,\right].
\end{align*}
This proves (\ref{french}). Now we need to show that
\begin{equation}
  \label{fr}
  P_{0}\left[\vec D^{(K)}=\vec i\,\right]=P_0\left[\vec V=\vec i\,\right].
\end{equation}
The proof is essentially the same as that of Proposition~2.2 of
\cite{BS06}.  At first, notice that, given
$\big(D^{(K)}_1,\dots,D^{(K)}_k\big)$, the distribution of $D^{(K)}_{k+1}$ depends
only on $D^{(K)}_k$. Therefore, the process $(D^{(K)}_k)_{0\le k\le
  K}$ is Markov, just as the process $(\wV_k)_{0\le k\le K}$.  Both
processes get absorbed after the first return to $0$. Then (\ref{fr})
will follow if we show that they have the same transition
probabilities and that $D^{(K)}_1$ has the same distribution as $\wV_1$.
Let $m\ge 1$ and $1\le k\le K-1$ or $m=k=0$.  Notice that
if the number $D^{(K)}_k$ of downcrossings of the edge $(K-k,K-k-1)$ prior to $T_K$
is $m$ then the number of upcrossings of the same edge prior to $T_K$
equals $m+1$.  Therefore, the number $D_{k+1}^{(K)}$ of downcrossings of the edge
$(K-k-1,K-k-2)$ prior to $T_K$ is equal to the number of failures in
$\big(Y^{(K-k-1)}_i\big)_{i\ge 1}$ before the $(m+1)$-st success,
which is $F^{(K-k-1)}_{m+1}$. On the other hand, if $\wV_k=m$ then, by
(\ref{VV}) and (\ref{nv}), $\wV_{k+1}= F^{(k)}_{m+1}$. But for all
$i,j\ge 0$ random variables $F^{(i)}_{m+1}$ and
$F^{(j)}_{m+1}$ have the same distribution. 
\end{proof}

\section{Law of large numbers and ballisticity}\label{Slln}
While (\ref{v}) gives the law of large numbers in the transient case
the renewal structure does not say anything about the recurrent case.
The following general result covers both the transient and the
recurrent case.
\begin{prop}\label{lln}
There is a deterministic $v\in [-1,1]$ such that $P_0$-a.s.\ $X_n/n\to v$ for
$n\to\infty$.
\end{prop}
\begin{proof}
  It can be shown exactly like in the proof of \cite[Theorem 13]{Ze05}
  that if $\sup_{n\ge 0}X_n=\infty$ a.s.\ then
\begin{eqnarray} \label{s1} \limsup_{n\to\infty}\frac{X_n}{n}&\le
  &\frac{1}{u_+}\quad\mbox{ a.s., where }\\ \nonumber
  u_+&:=&\sum_{j\ge 1}P_0[T_{j+1}-T_j\ge j]\in[1,\infty],\quad\mbox{and}\\
  \liminf_{n\to\infty}\frac{X_n}{n}&\ge &\frac{1}{u_+}\quad\mbox{a.s.\
    if $u_+<\infty$ }\label{s2}
\end{eqnarray}
see the last line on p.\ 113 and the first line on p.\ 114 of \cite{Ze05}. 
Similarly, by symmetry, if $\inf_{n\ge 0} X_n=-\infty$ a.s.\ then
\begin{eqnarray}\label{s3}
  \liminf_{n\to\infty}\frac{X_n}{n}&\ge &\frac{-1}{u_-}\quad\mbox{ a.s., where }\\
  u_-&:=&\sum_{j\ge 1}P_0[T_{-j-1}-T_{-j}\ge
  j]\in[1,\infty],\quad\mbox{and }
\nonumber\\
  \limsup_{n\to\infty}\frac{X_n}{n}&\le
  &\frac{-1}{u_-}\quad\mbox{a.s.\ 
if $u_-<\infty$. }\label{s4}
\end{eqnarray}
Now due to Theorem \ref{rt0} there are only three cases: Either the
walk is transient to the right or it is transient to the left or it
is recurrent.  Consider the case of transience to the right.  If
$u_+<\infty$ then $\lim_nX_n/n=1/u_+$ follows directly from (\ref{s1})
and (\ref{s2}).  If $u_+=\infty$ then $\lim_nX_n/n=0$ follows from
(\ref{s1}) and $\inf_n X_n>-\infty$.  Transience to the left is
treated analogously. In the case of recurrence we have a.s.\ both
$\sup_nX_n=\infty$ and $\inf_nX_n=-\infty$. Hence both $u_+$ and $u_-$
are infinite due to (\ref{s2}) and (\ref{s4}), respectively. Therefore, by
(\ref{s1}) and (\ref{s3}), $\lim_nX_n/n=0$.
\end{proof}
\begin{lemma}\label{jem}
  Let  $(Z_k)_{k\ge
    0}$ be a $(\geo,\nu)$-bran\-ching process, where $\nu$ is the
  distribution of $\eta_k:=F^{(k)}_M-M+1$, and recall definitions  (\ref{wZ}),  (\ref{VV}) and (\ref{nv}). Then
\begin{equation}\label{mm}
E_0[\tilde V]<\infty \Longleftrightarrow E_0[\tilde Z]<\infty,\quad
\mbox{where}\quad \tilde V:=\sum_{k\ge 0} \tilde V_k\quad\mbox{and}\quad
\tilde Z:=\sum_{k\ge 0} \tilde Z_k.
\end{equation}
\end{lemma}
\begin{proof} 
  As an intermediate step we first consider the auxiliary Markov chain
  $(W_k)_{k\ge 0}$ defined by
\begin{equation}\label{WW}
 W_0:=0,\quad W_{k+1}:=F^{(k)}_{(W_k+1)\vee M}.
\end{equation}
This is a branching process with migration in the following sense:
At each step, it exhibits two types of behavior: 1) if $W_k\ge M-1$ then one
particle immigrates and then all $W_k+1$ particles reproduce; 2) if
$W_k<M-1$ then $M-W_k$ particles immigrate and then all $M$ particles reproduce.

We shall first establish the equivalence
\begin{equation}\label{mm2}
E_0[\tilde V]<\infty  \Longleftrightarrow E_0[\tilde W]<\infty.
\end{equation}
where, as usual,
\begin{equation}
  \label{nw}
N^{(W)}:=\inf\{k\ge 1\mid W_k=0\},\quad 
  \wW_k:=W_k\won_{\{k<N^{(W)}\}},\ \text{and}\quad
\tilde W:=\sum_{k\ge 0} \tilde W_k.
\end{equation}
Comparing definitions (\ref{VV}) and (\ref{WW}) we see that $\tilde
V_k\le \tilde W_k$ for all $k$, which yields the implication
$\Leftarrow$ in (\ref{mm2}).  For the reverse implication, assume that
$E_0[\tilde V]$ is finite. Since $N^{(V)}\le\tilde V+1$ this implies
that $(V_k)_{k\ge 0}$ is positive recurrent.
The following
lemma, whose proof is postponed, will help us to compare $(V_k)_{k\ge
  0}$ and $(W_k)_{k\ge 0}$.
\begin{lemma}\label{fini} Let $K$ be the transition matrix of a
  positive recurrent Markov chain with state space $\N_0$ 
  and invariant distribution $\pi$. Assume also that all entries of $K$
  are strictly positive.  Fix a state $j\in \N_0$ and a finite set
  $J\subset \N_0\setminus \{j\}$. Modify a finite number of rows of
  $K$ by setting
  \begin{equation*}
    \wK(i,\cdot):=
    \begin{cases}
     K(i,\cdot),&\text{if }i\not\in J;\\K(j,\cdot),&\text{if }i\in J. 
    \end{cases}
  \end{equation*}
 Then a Markov chain with the transition matrix $\wK$ is also
  positive recurrent and its unique invariant probability distribution
  $\wpi$ satisfies $\wpi(n)\le \con{cc}\pi(n)$ for
all $n\in \N_0$.
\end{lemma}
If we let $K$ be the transition matrix of
the Markov chain $(V_k)_{k\ge 0}$ and set $j=M-1$ and
$J=\{0,1,\ldots,M-2\}$ then $\wK$ defined in Lemma \ref{fini} is the
transition matrix of $(W_k)_{k\ge 0}$. Moreover, all entries of this
$K$ are strictly positive due to (\ref{elli}).  Consequently, we may
apply Lemma \ref{fini} and get that $(W_k)_{k\ge 0}$ is positive
recurrent and its invariant probability distribution $\wpi$ is bounded
above by a multiple $c_{\ref{cc}}\pi$ of the invariant probability
distribution $\pi$ of $(V_k)_{k\ge 0}$. By Theorem 5.4.3 of
\cite{durr}, $\pi$ and $\wpi$ can be represented as
$\pi=\rho/E_0[N^{(V)}]$ and $\wpi=\wrho/E_0[N^{(W)}]$, where for $s\in \N_0$,
\[ \rho(s):=E_0\Bigg[\sum_{k=0}^{ N^{(V)}-1}\won_{\{
    V_k=s\}}\Bigg]\quad\mbox{and}\quad 
\wrho(s):=E_0\Bigg[\sum_{k=0}^{ N^{(W)}-1}\won_{\{
    W_k=s\}}\Bigg].
\]
Therefore, also $\wrho\le\con{til}\rho$. However,
\begin{equation}\label{wir}
E_0[ 
\tilde W]= E_0\Bigg[\sum_{k= 0}^{N^{(W)}-1}\sum_{s\ge 0}s\won_{\{W_k=s\}}\Bigg]
\ =\ \sum_{s\ge 0}s\,\wrho(s)
\end{equation}
and, similarly, $E_0[\tilde V] = \sum_{s}s\rho(s).$ Consequently,
$E_0[\tilde W]\le c_{\ref{til}}E_0[\tilde V]$. This gives the
implication $\Rightarrow$ in (\ref{mm2}). Next, we show that
\begin{equation}\label{mm3}
  E_0[ \tilde W]<\infty \Longleftrightarrow E_0[\tilde Z]<\infty.
\end{equation}
As in Lemma \ref{deco1} we decompose the process
$(W_k)_{k\ge 0}$ into two components.
\begin{lemma}\label{deco2}
For $k\ge 0$ let
$Z'_k:=W_{k+1}-F^{(k)}_M.$
Then $(Z'_k)_{k\ge 0}$ is a $(\geo,\nu)$-branching process, where
$\nu$ is the common distribution of $\eta_k:=F_M^{(k)}-M+1$ under
$P_1$.
\end{lemma}
The proof of Lemma \ref{deco2} is almost identical to the one of Lemma
\ref{deco1} and, thus, is omitted.  

Since $F^{(k)}_M\ge 0$, we immediately obtain $\tilde Z'_k\le \tilde
W_{k+1}$, where $\tilde Z'_k$ is defined by replacing $W$ in (\ref{nw}) by $Z'$.
By Lemma \ref{deco2}, $(\wZ_k)_k$ and $(\wZ'_k)_k$ have the same distribution. Therefore $E_0[\wZ]=E_0[\wZ']\le E_0[\tilde W]$, which
yields the implication $\Rightarrow$ in (\ref{mm3}). For the opposite
direction assume that $E_0[\wZ]<\infty$. Then, as in the proof of
(\ref{mm2}), $(Z'_k)_{k\ge 0}$ is positive recurrent and, by the
equivalent of (\ref{wir}), its invariant distribution, say $\pi'$, has
a finite mean.  Since $Z'_k$ and $F^{(k)}_M$ are independent, it
follows from Lemma \ref{deco2} that the convolution of $\pi'$ and the
distribution of $F^{(k)}_M$ is invariant for $(W_k)_{k\ge 0}$.  This
convolution has a finite mean as well, which implies, as in (\ref{wir}), that
$E_0[\tilde W]$ is finite as well.  This concludes the proof of 
(\ref{mm3}). The statement of the lemma now follows from (\ref{mm2}) and (\ref{mm3}). 
\end{proof}
\begin{proof}[Proof of Lemma \ref{fini}]
  It suffices to consider the case in which $J$ has only one element, i.e.\
  $J=\{i\}$ for some $i\ne j$. The full statement then follows by induction, changing
  one row at a time.
  Let $(\zeta_k)_{k\ge 0}$
  and $(\wzeta_k)_{k\ge 0}$ be Markov chains with transition matrices
  $K$ and $\wK$, respectively.  Their initial point will be denoted by
  a subscript of $P$ and $E$.  Since all the entries of $\wK$ are
  strictly positive, $\wK$ is irreducible.  It is recurrent, since its
  state $i$ is recurrent.  Indeed, $P_i[\exists k\ge 1:\wzeta_k=i]=
  P_j[\exists k\ge 1:\wzeta_k=i]$ because of $\wK(i,\cdot)=K(j,\cdot)=
  \wK(j,\cdot)$.  Moreover, since $(\zeta_k)_{k\ge 0}$ and
  $(\wzeta_k)_{k\ge 0}$ are indistinguishable as long as they do not
  touch $i$, i.e.\ since $K(s,\cdot)=\wK(s,\cdot)$ for all $s\ne i$,
  we can switch from the process $(\wzeta_k)_{k\ge 0}$ to
  $(\zeta_k)_{k\ge 0}$ and obtain that $P_i[\exists k\ge
  1:\wzeta_k=i]=P_j[\exists k\ge 1:\zeta_k=i]$, which is equal to 1
  because $(\zeta_k)_{k\ge 0}$ is recurrent.

  Similarly, one can show that $(\wzeta_k)_{k\ge 0}$ is also positive
  recurrent.  Define the hitting time $\si:=\inf\{k\ge
  1\mid \zeta_{k}=i\}$ for $(\zeta_k)_{k\ge 0}$ and analogously $\wsi$
  for $(\wzeta_k)_{k\ge 0}$.  Then, again by \cite[Theorem
  5.4.3]{durr}, $\rho$ and $\wrho$, defined by
\[
  \rho(s):=E_i\left[\sum_{k=0}^{\si-1}\won_{\{
      \zeta_k=s\}}\right]\quad\mbox{and}\quad\wrho(s):=
  E_i\left[\sum_{k=0}^{\wsi-1}\won_{\{
      \wzeta_k=s\}}\right],\qquad s\in \N_0,
\]
are invariant measures for $K$ and $\wK$, respectively. Using the
relations between $K$ and $\wK$ as above, we have for all $s\in \N_0$,
\begin{equation}\label{na}\wrho(s)=E_j\left[\sum_{k=0}^{\wsi-1}\won_{\{
    \wzeta_k=s\}}\right]\ =\ E_j\left[\sum_{k=0}^{\si-1}\won_{\{\zeta_k=s\}}\right].
\end{equation}
On the other hand, for all $s\in \N_0$,
\[\rho(s)\ge E_i\left[\sum_{k=1}^{\si-1}\won_{\{\zeta_1=j,\
    \zeta_k=s\}}\right]\ =\ K(i,j)\ E_j\left[\sum_{k=0}^{\si-1}\won_{\{
    \zeta_k=s\}}\right]\stackrel{(\ref{na})}{=}K(i,j)\wrho(s).\]
Since $(\zeta_k)_{k\ge 0}$ is positive recurrent, $\rho$'s total mass,
$E_i[\si]$, is finite.  Consequently, by the above and since
$K(i,j)>0$, $\wrho$'s total mass, $E_i[\wsi]$, is finite as well.
Therefore, $(\wzeta_k)_{k\ge 0}$ is positive recurrent and its
invariant measure $\wpi$ satisfies $\wpi\le c_{\ref{cc}}\pi$ with
$c_{\ref{cc}}:=E_i[\si]/(E_i[\wsi]K(i,j))$.
\end{proof}
The following lemma is the counterpart of Lemma \ref{th}. 
\begin{lemma}\label{th2}
 Let $\nu$ be the distribution of $F_M^{(0)}-M+1$ under $P_0$. Then
  $\theta$ defined in (\ref{theta}) for the $(\geo,\nu)$-branching
  process is equal to $1-\delta$.
\end{lemma}
\begin{proof} 
As in the proof of Lemma \ref{th}, $\theta=E_0[F_{M}^{(0)}]-M+1$ since $b=1$.
Switching failures and successes in (\ref{hello}) and replacing $\om(x,i)$  
  by $1-\om(x,i)$ yields 
  $E_0[F_{M}^{(0)}]-M=-\delta$.
\end{proof}
\begin{proof}[Proof of Theorem \ref{v0}]
  The first statement of the theorem, the existence of the velocity $v$,
  is just Proposition \ref{lln}, or (\ref{v}) in the transient case.
  If $\delta\in[-1,1]$ then the walk is recurrent by Theorem \ref{rt0}
  and therefore $v=0$.

Now let $|\delta|>1$. Without loss of generality we may assume
$\delta>1$. Then the walk is transient to the right by Theorem
\ref{rt0}.  By (\ref{first}), $v>0$ iff $E_0[\tau_2-\tau_1]<\infty$.
By Lemma \ref{mom} with $p=1$ this is the case iff
$E_0\big[\sum_{k\ge 0}D_k\big]<\infty$. By Lemma \ref{ffs} this holds iff
$E_0\big[\sum_{k\ge 0}\tilde V_k\big]<\infty$. Due to Lemma
\ref{jem} this is true iff $E_0\big[\sum_{k\ge 0}\tilde Z_k\big]$ is
finite. By the second statement in Corollary \ref{AA} this holds iff
$\theta<-1$.  
Thus, by Lemma \ref{th2},
$v>0$ iff $\delta>2$.
\end{proof}
\section{Central limit theorem}\label{sclt}
\begin{lemma}\label{lclt}
  Let $(\wV_k)_{k\ge 0}$ be defined by (\ref{nv}). Then $\delta>4$
  implies that the random variable
  $\tilde{V}:=\sum_{k\ge 0}\tilde V_k$ has a finite second moment. 
\end{lemma}
\begin{proof}
  Let $(W_k)_{k\ge 0}$ and $\tilde{W}$ be defined by (\ref{WW}) and (\ref{nw}), respectively, using the
  same sequences $(F^{(k)}_i)_{i\ge 0},\ k\in \N_0$,  as for the process
  $(V_k)_{k\ge 0}$.
Since $V_k\le W_k$, we have $ N^{(V)}\le N^{(W)}$ and $E_0[\tilde{V}^2]\le
E_0[\tilde{W}^2]$. We shall prove that the latter is finite.
By Minkowski's inequality we have
\[
\left(E_0[\tilde{W}^2]\right)^{1/2}=
\bigg(E_0\bigg[\bigg(\sum_{k\ge 0}W_k\won_{\{ N^{(W)}>k\}}\bigg)^2\bigg]\bigg)^{1/2}
\le \sum_{k\ge 0} 
\left(E_0\left[W_k^2\won_{\{ N^{(W)}>k\}}\right]\right)^{1/2}.
\]
From Lemma \ref{deco2} we see that $Z_k:=W_{k+1}-F_M^{(k)}$ defines a
$(\geo,\nu)$-branching process, where $\nu$ is the distribution of
$F_M^{(0)}-M+1$.  Therefore, $W^2_k\le 2(Z_{k-1}^2+(F^{(k-1)}_M)^2)$.
Combining this with the fact that $(a+b)^{1/2}\le a^{1/2}+b^{1/2}$ for
all $a,b\ge 0$, we get
\begin{align*}
  \left(E_0[\tilde{W}^2]\right)^{1/2}
&\le \sqrt{2}\,\sum_{k=1}^\infty
  \left(E_0\left[Z_{k-1}^2\won_{\{ N^{(W)}>k\}}\right]+E_0\big[\big(F^{(k-1)}_M\big
      )^2\won_{\{ N^{(W)}>k\}}\big]\right)^{1/2}\\ 
&\le \sqrt{2}\,\sum_{k=1}^\infty\left[
  \Big(E_0\left[Z_{k-1}^2\won_{\{ N^{(W)}>k\}}\right]\Big)^{1/2}+
\left(E_0\big[\big(F^{(k-1)}_M\big
      )^2\won_{\{ N^{(W)}>k\}}\big]\right)^{1/2}\right].
\end{align*}
Applying H\"older's inequality with $1/\alpha+1/\alpha'=1$,
$\alpha>1$, we obtain
\begin{multline}
  \left(E_0[\tilde{W}^2]\right)^{1/2}\le \sqrt{2}\sum_{k=1}^\infty
  \left(E_0\left[Z_{k-1}^{2\alpha}\right]\right)^{1/(2\alpha)}
  \big(P_0[ N^{(W)}>k]\big)^{1/(2\alpha')}\\+
  \sqrt{2}E_0\big[\big(F^{(0)}_M\big)^{2\alpha}\big]^{1/(2\alpha)}\sum_{k=1}^\infty\big
  (P_0[ N^{(W)}>k]\big)^{1/(2\alpha')}.\label{split}
\end{multline}
We are going to show that 
\begin{itemize}
\item [(i)] for every $\epsilon\in (0,\delta-4)$ there is a constant
  $\con{ed}(\epsilon,\delta)$ such that \[P_0[ N^{(W)}>k]\le
  c_{\ref{ed}}k^{-\delta+\epsilon}\quad\text{for
    all $k\in\N$};\]
\item [(ii)] for each $\ell\in\N$ there is a constant
  $\con{ell}(\ell)$ such that $E_0\left[Z_k^\ell\right]\le
  c_{\ref{ell}}(\ell)k^\ell$ for all $k\in\N_0$.
\end{itemize}
Let us assume (i) and (ii) for the moment and see that both series in
the right hand side of (\ref{split}) are finite. Choose
$\alpha'\in(1,\delta/4)$ so that $\alpha=\alpha'/(\alpha'-1)$ is an
integer and let $\epsilon= (\delta-4\alpha')/2$. Then by (i) and (ii)
for all $k\ge 1$,
\[\left(E_0\left[Z_{k-1}^{2\alpha}\right]\right)^{1/(2\alpha)}
\big(P_0[ N^{(W)}>k]\big)^{1/(2\alpha')} \le
\con{ad}(\alpha',\delta)k^{1-(\delta-\epsilon)/(2\alpha')}=c_{\ref{ad}}k^{-\delta/(4\alpha')}.
\]
Since $\delta/(4\alpha')>1$, the first series in the right hand side
of (\ref{split}) converges. It is obvious now that for the same choice
of $\alpha'$ and $\epsilon$ the second series in the right hand side
of (\ref{split}) also converges.  Therefore we only need to prove (i)
and (ii).

{\em Proof of (i).} Observe that $W_k$ is zero if and only if both
$Z_{k-1}$ and $F^{(k-1)}_M$ are equal to zero. Set $N_0:=0$ and consider the
times $ N_i:=\inf\{k> N_{i-1}\,|\,Z_k=0\}$, $i\in \N$, when the process $(Z_k)_{k\ge 1}$ dies
out.
Due to Lemma~\ref{th2} and $\delta>4$  the parameter $\theta$ for
the process $(Z_k)_{k\ge 0}$ satisfies
\begin{equation}\label{tm3}
  \theta=1-\delta<-3.
\end{equation}
In particular, Corollary \ref{AA} implies that the
process $(Z_k)_{k\ge 0}$ is positive recurrent.  Therefore,
all $ N_i$, $i\in\N$, are a.s.\ finite and
$( N_i- N_{i-1})_{i\in\N}$ is i.i.d.. We are interested in the sequence $(F^{( N_i)}_M)_{i\in\N}.$
 Observe that $\big(F^{(n)}_M\big)_{n\ge 0}$ is
i.i.d.\ and, by construction, $F^{(n)}_M$ is independent from ${\cal
  F}_n=\sigma(\{Z_k,k\le n\},\{F^{(j)}_M,j<n\})$.  It is then
straightforward to check that each $F^{( N_i)}_M$ is
independent from ${\cal F}_{ N_i}$, $i\in\N$,
and 
$\big(F^{( N_i)}_M\big)_{i\ge 1}$ is i.i.d.. 
Let
$\kappa:=\inf\big\{i\ge 1\,|\, F^{( N_i)}_M=0\big\}$. Then $\kappa$ has the
geometric distribution on $\N$ with parameter $p:=P[F^{(0)}_M=0]\in(0,1)$.
Therefore, 
$ N^{(W)}= N_{\kappa}=\sum_{i=1}^\kappa ( N_i- N_{i-1})$ and 
\begin{align}
  P_0[ N^{(W)}>k]&=P_0\left[ N_\kappa^{\delta-\eps}>k^{\delta-\eps}\right]\le
  \frac{1}{k^{\delta-\epsilon}}\sum_{m\ge 1}
  E_0\left[ N_m^{\delta-\epsilon}\won_{\{\kappa=m\}}\right]\nonumber\\&\le 
\frac{1}{k^{\delta-\epsilon}}\sum_{m\ge 1}
  \left(E_0\left[ N_m^{(\delta-\epsilon)\beta}\right]\right)^{1/\beta}
  \left(P_0[\kappa=m]\right)^{1/\beta'},\label{kap}
\end{align}
where $\beta\in(1, \delta/(\delta-\epsilon))$, $\beta'=\beta/(\beta-1)$. Denote
$(\delta-\epsilon)\beta$ by $\gamma$. Writing
$ N_m$ as $\sum_{i=1}^m( N_i- N_{i-1})$ and using Minkowski's
inequality we get 
$(E_0[N_m^\gamma])^{1/\gamma}\le   m  (E_0[ N_1^\gamma])^{1/\gamma}.$
Substituting this into (\ref{kap}) we obtain
\begin{equation*}
  P_0[ N^{(W)}>k]\le \frac{\left(E_0\left[
        N_1^\gamma\right]\right)^{1/\beta}} {k^{\delta-\epsilon}}\sum_{m=1}^\infty
  m^{\delta-\epsilon}\left((1-p)^{m-1}p\right)^{1/\beta'}\ \le\  
  \frac{\left(E_0\left[
        N_1^\gamma\right]\right)^{1/\beta}}{k^{\delta-\epsilon}}\, \con{eve}(\eps,\delta).
\end{equation*}
From part  (iv) of Theorem~A and (\ref{tm3}) we know that $P_0[N_1>k]\sim
c_{\ref{four}} k^{-\delta}$.  Therefore, and since $\gamma<\delta$ by assumption, $P_0[N_1^\ga>k]$ is summable in $k$. Consequently,
$E_0\left[ N_1^\gamma\right]=\con{ei}(\epsilon,\delta)<\infty$.
This implies (i).

{\em Proof of (ii).} The proof can be easily done by induction in $\ell$.
The statement is trivial for $\ell=0$. (Here $0^0=1$.)
Assume now $\ell\ge 1$ and that for each
  $j\in\{0,1,2,\dots,\ell-1\}$ there is a constant $c_{\ref{ell}}(j)$ such that
  $E_0\left[Z_k^j\right]\le c_{\ref{ell}}(j)k^j$ for all $k\in \N_0$.  
Using that $(Z_k)_k$ is a $(\geo,\nu)$-branching process  we have for all
  $k\in\N_0$ and $n\in\N$,
  \begin{align}\nonumber
    E_0\left[Z_{k+1}^\ell\right]
&=E_0\left[\left(\xi^{(k+1)}_1+
        \xi^{(k+1)}_2+\dots+\xi^{(k+1)}_{Z_k+\eta_k}\right)^\ell\right]\\
    &=E_0\left[\left(\xi^{(k+1)}_1+ \nonumber
        \dots+\xi^{(k+1)}_{Z_k+\eta_k}\right)^\ell\won_{\{Z_k+\eta_k\le \ell\}}\right]\\
    &\makebox[3cm]{\ }
+\ \sum_{n>\ell}
    E_0\left[\left(\xi^{(k+1)}_1+
        \dots+\xi^{(k+1)}_{Z_k+\eta_k}\right)^\ell\won_{\{Z_k+\eta_k=n\}}\right] 
\nonumber \\
    &\le \label{post} E_0\left[\left(\xi^{(0)}_1+ 
        \dots+\xi^{(0)}_{\ell}\right)^\ell\right]
    \\\label{ser}
    &\makebox[3cm]{\ }+\sum_{n>\ell}
    E_0\left[\left(\xi^{(0)}_1+
        \dots+\xi^{(0)}_{n}\right)^\ell\right]P_0\left[Z_k+\eta_k=n\right].
 \end{align}
  Observe that the expectation in (\ref{post}) is bounded by a constant
  $\con{con}(\ell)$. To control the series in (\ref{ser}) we use the
  following lemma, whose proof is postponed until after the end of the
  present proof.
\begin{lemma}\label{multi}
  Let $(\xi_i)_{i\in\N}$ be non-negative i.i.d.\ random variables such
  that $E[\xi_1]=1$ and $E\left[\xi_1^\ell\right]<\infty$ for some
  positive integer $\ell$. Then there is a constant
  $\con{lll}$ such that for all $n> \ell$,
\[E\left[(\xi_1+\xi_2+\dots+\xi_n)^\ell\right]\le
n^\ell+c_{\ref{lll}}\sum_{m=1}^{\ell-1} n^m. 
\]
\end{lemma}
Applying this lemma to the series in (\ref{ser}) we obtain
\begin{eqnarray*} E_0\left[Z_{k+1}^\ell\right]&\le &
  c_{\ref{con}}+\sum_{n>\ell}\left(n^\ell+c_{\ref{lll}} 
    \sum_{m=1}^{\ell-1}n^m\right)
  P_0\left[Z_k+\eta_k=n\right]\\
  &\le&c_{\ref{con}}+E_0\left[(Z_k+\eta_k)^\ell\right]+c_{\ref{lll}}
  \sum_{m=1}^{\ell-1}
  E_0\left[(Z_k+\eta_k)^m\right]\\
  &\le&E_0\left[Z_k^\ell\right]+\sum_{m=0}^{\ell-1}\con{coco}(m,\ell)E_0[Z_k^m].
\end{eqnarray*}
  By the induction hypothesis we conclude that 
$E_0\left[Z_{k+1}^\ell-Z_k^\ell\right]\le \con{caci}k^{\ell-1}.$
Summation over $k$
implies (ii) and finishes the proof of  Lemma \ref{lclt}.
\end{proof}
\begin{proof}[Proof of Lemma \ref{multi}.]
By expanding and using independence,
\[
E\left[(\xi_1+\xi_2+\dots+\xi_n)^\ell\right]
=\sum_{m=1}^\ell{n \choose m} \sum_{\substack{\ell_1+\ldots+\ell_m=\ell\\1\le \ell_1,
      \ldots, \ell_m}} \dbinom
  {\ell}{\ell_1,\dots,\ell_m}\prod_{j=1}^m
  E\left[\xi_1^{\ell_j}\right].
\]
Since $E[\xi_1]=1$ the summand for $m=\ell$ is
equal to $n!/(n-\ell)!\le n^\ell$.  For the other terms we estimate the factor ${n\choose m}$ 
from above by $n^m$ and define the constant
\[c_{\ref{lll}}:=\max_{1\le m<\ell}\sum_{\substack{\ell_1+\dots+\ell_m=\ell\\1\le \ell_1,
      \ldots, \ell_m}} \dbinom
  {\ell}{\ell_1,\ldots,\ell_m}\prod_{j=1}^m
  E\left[\xi_1^{\ell_j}\right].
\]
\end{proof}
\begin{proof}[Proof of Theorem \ref{clt}]
  Let $|\delta|>4$. By symmetry we may assume without loss of
  generality that $\delta>4$.  Then, by Lemma \ref{lclt},
  $E_0[(\sum_{k\ge 1}\widetilde V_k)^2]<\infty$. By Lemma \ref{ffs} this
  implies $E_0[(\sum_k D_k)^2]<\infty$. Lemma \ref{mom} for  $p=2$
  gives $E_0[(\tau_2-\tau_1)^2]<\infty$.  This implies the claim, as
  indicated in (\ref{second}) and (\ref{si2}).
\end{proof}

\section{Further remarks, a multi-dimensional example, and open questions}\label{mult}
\begin{rem}[Permuting cookies]\label{order} {\rm Note that
    permuting cookies within the cookie piles, i.e.\ replacing
    $(\om(\cdot,i))_{i\ge 1}$ by $\om(\cdot,\pi(i))_{i\ge 1}$, where
    $\pi:\N\to\N$ is a permutation, does not change $\delta$, defined
    in (\ref{D}). Therefore, permuting cookies does not change the
    classification of the walk as described in Theorems \ref{rt0} and
    \ref{v0}.  This fact can also be seen as follows without using the
    results in  Theorem \ref{A} from the literature.

    At first, observe that we may assume without
    loss of generality that $\pi$ is a finite permutation, since all cookies
    in a cookie pile except for a finite number are placebo cookies.  We may
    even assume that all $i>M$ are fixed points for $\pi$.  Indeed,
    otherwise just replace the original $M$ with $\max\{\pi(i)\mid i\le
    M\}$.  To argue, for example, that permuting cookies does not turn a
    walk which is recurrent from the right into one which is not recurrent from 
    the right or vice versa, recall the definition of $S_M^{(k)}$
    in (\ref{sucks})
    and of $\eta_k$ in Lemma \ref{tree2} and denote by $\si^{(k)}_M\ge
    M$ the index of the $M$-th failure in the sequence
    $(Y_i^{(k)})_{i\ge 1}$. Then $\eta_k$ can be written as
\begin{equation}\label{perm}
  \eta_k=\sum_{i=1}^{\si^{(k)}_M}Y_i^{(k)}=\sum_{i=1}^MY_i^{(k)}
+\sum_{i=M+1}^{\si^{(k)}_M}Y_i^{(k)}=
  \sum_{i=1}^MY_{\pi(i)}^{(k)}+\sum_{i=M+1}^{\si^{(k)}_M}Y_{\pi(i)}^{(k)},
\end{equation} 
where we used in the last step for the first sum that $\pi$ maps
$\{1,\ldots,M\}$ onto itself and for the second sum that all $i>M$ are
fixed points for $\pi$. For the same reason, $\si^{(k)}_M\ge M$ is also
the index of the $M$-th failure in the permuted sequence
$Y_{\pi(i)}^{(k)}$.  
This
shows that applying $\pi$ does not change the distribution $\nu$ of
$\eta_k$. Therefore, recurrence to the right, as characterized in
Lemma \ref{tree2}, is invariant under permutations.  Since the proof
of Lemma \ref{tree2} did not 
use any results from Section~\ref{lit}, this proof is self-contained.

Similarly, one can show that the positivity of speed is invariant
under permutations. Indeed, in the proof of Theorem \ref{v0} we have
shown that $v>0$ iff $E_0[\sum_{k\ge 0}\tilde Z_k]<\infty$. By the
above argument, the distribution of $(\tilde Z_k)_{k\ge 0}$ remains unchanged
under permutations.  }
\end{rem}
\begin{rem}[Higher dimensions] {\rm 
Multi-dimensional ERW with cookies that induce a bias with
a non-negative projection in some fixed direction were considered in
\cite{bewi}, \cite{K03}, \cite{K05}, \cite{Ze06}, \cite{HH06} and \cite{BeR}.  A
special non-stationary cookie environment with two types of cookies
pointing into opposite directions was studied in \cite{ABK08}.
However, to the best of our knowledge so far there are no 
criteria for recurrence, transience or ballistic behavior
of ERWs in i.i.d.\ environments  of
``positive'' and ``negative'' cookies in higher dimensions.  To us, it
is not even clear how such criteria could look like.  In the following example we
shall indicate that the situation cannot be as simple as in one
dimension.
We shall show that, unlike for $d=1$ (see Remark \ref{order}),
permuting cookies in higher dimensions may change the sign of the velocity.
}
\end{rem}
\begin{ex} 
  \rm{ Let  $d\ge 4$.
    Denote by $\om(x,e,i)$ the probability for the ERW to jump from
    $x\in\Z^d$ to the nearest neighbor $x+e$ upon the $i$-th visit of
    $x$. Fix $0<\eps<1/(2d)$ and consider the two deterministic
    cookie environments $\om_k$, $k=1,2,$ defined by
\[\om_k(x,\pm e_j,i):=\frac{1}{2d}\pm\eps(-1)^{i+k}\won_{\{j=1;\,i\in\{1,2\}\}},
\quad x\in\Z^d,\ 1\le j\le d,\ i\ge 1,\] where $\{e_1,\ldots,e_d\}$ is
the canonical basis of $\Z^d$. In both environments, $\om_1$
and $\om_2$, there are $M=2$ cookies per site.  In the environment
$\om_1$ the walk experiences a drift into direction $e_1$ upon the
first visit to a site and an equal drift into the opposite direction
$-e_1$ upon the second visit.  Otherwise, it behaves just like a
simple symmetric random walk.  We shall show:
\begin{equation}\label{hin}
\mbox{There is some $v>0$ such that}\quad  
\lim_{n\to\infty}\frac{X_n}{n}= ve_1\quad\mbox{$P_{0,
\om_1}$-a.s.}.
\end{equation}
The environment $\om_2$ is obtained from $\om_1$ by permuting the two
cookies at each site.  By symmetry, we obtain from (\ref{hin}) that for
the same $v>0$, we have $P_{0,\om_2}$-a.s.\ $\lim_n X_n/n=-ve_1$.
Thus, in this example, permuting two cookies reverses the direction of
the speed.

For the proof of (\ref{hin}) denote by $R_{n,1}:=\{X_m\mid m<n\}$ the
range of the walk before time $n\in\N$ and by
$R_{n,2}:=\{x\in\Z^{d}\mid \exists k<m<n\ X_k=x=X_m\}$ the set of
vertices which have been visited at least twice before time $n$.
It is easy to see that both $\#R_{n,1}$ and $\#R_{n,2}$ tend to $\infty$ as $n\to\infty$.
Coupling
$(X_n)_{n\ge 0}$ in the natural way to a simple symmetric random walk
$(S_n)_{n\ge 0}$ on $\Z^d$ yields that $(X_n)_{n\ge 0}$ can be
represented as
\[X_n=S_n+2e_1a_n-2e_1b_n,\quad\mbox{where}
\quad a_n=\sum_{i=1}^{\#R_{n,1}}Y_{i,1}\quad\mbox{and}\quad
b_n=\sum_{i=1}^{\#R_{n,2}}Y_{i,2}
\] 
and $Y_{i,j}\  (i\ge 1, j=1,2)$ are independent and Bernoulli 
distributed with parameter $\eps$. Consequently,
\begin{eqnarray}\nonumber
  \frac{X_n\cdot e_1}{n}&=&\frac{S_n\cdot e_1}{n}+\frac{2a_n}{\#R_{n,1}} 
\frac{\#R_{n,1}}{n}
  -\frac{2b_n}{\#R_{n,2}} \frac{\#R_{n,2}}{n}\\
  &=&\frac{S_n\cdot e_1}{n}+\frac{2a_n}{\#R_{n,1}} \frac{\#R_{n,1}-\#R_{n,2}}{n}
  +\left(\frac{2a_n}{\#R_{n,1}}-\frac{2b_n}{\#R_{n,2}}\right)
  \frac{\#R_{n,2}}{n}.
\label{rn}
\end{eqnarray}
The first term on the right hand side of (\ref{rn}) tends to zero
$P_{0,\om_1}$-a.s..  The same holds for the last term in (\ref{rn})
since by the strong law of large numbers both $a_n/\#R_{n,1}$ and
$b_n/\#R_{n,2}$ tend to $\eps$ as $n\to\infty$ and $\#R_{n,2}/n$ is
bounded.  Therefore,
\begin{equation}\label{summ}
\liminf_{n\to\infty}\frac{X_n\cdot e_1}{n}\ge 2\eps \liminf_{n\to\infty}
\frac{\#R_{n,1}-\#R_{n,2}}{n}
\end{equation}
To bound the right hand side of (\ref{summ}) from below we introduce
the projection $\pi:\Z^d\to\Z^{d-1}$ defined by
$\pi(x_1,x_2,\ldots,x_d):=(x_2,\ldots,x_d)$ onto the subspace spanned
by $e_2,\ldots,e_d$ and consider the process $(X'_n)_{n\ge 0}$ defined
by $X'_n:=\pi(X_n)$.  Under $P_{0,\om_1}$, $(X'_n)_{n\ge 0}$
is a simple symmetric random walk on $\Z^{d-1}$ with holding times
which are i.i.d.\ and geometrically distributed with parameter
$1-1/d$.  As above, denote by $R'_{n,1}:=\{X'_m\mid m<n\}$ the
range of this walk before time $n\in\N$ and by
$R'_{n,2}:=\{x\in\Z^{d-1}\mid \exists k<m<n\ X'_k=x=X'_m\}$ the set of
vertices which have been visited at least twice before time $n$ by
$(X'_n)_n$.

Returning to the right hand side of (\ref{summ}) we first note that
$\#R_{n,1}-\#R_{n,2}=\#(R_{n,1}\backslash R_{n,2})$ is the number of
vertices which have been visited exactly once by $(X_n)_n$ before time
$n$.  This number is greater than or equal to the number
$\#(R'_{n,1}\backslash R'_{n,2})$ of vertices which have been visited
exactly once by the projected walk $(X'_n)_n$ before time $n$.
According to \cite{Pi74} this last number satisfies a strong law of
large numbers and grows like $p^2n$, where $p$ is the probability that
$(X'_n)_n$ never returns to its starting point. Since $d\ge 4$ and
simple symmetric random walk in three or more dimensions is transient,
we have $p>0$.  Therefore, by (\ref{summ}),
\begin{equation}\label{ski}
\liminf_{n\to\infty}\frac{X_n\cdot e_1}{n}\ge 2\eps p^2>0.
\end{equation}
In particular, $(X_n\cdot e_1)_n$ is transient to the right. Using a
renewal structure like in \cite{SZ99} for RWRE and in \cite{BeR} for
once-ERW and as outlined in Section \ref{rs} this implies that
$(X_n\cdot e_1)_n$ satisfies a strong law of large numbers under
$P_{0,\om_1}$, i.e.\ $P_{0,\om_1}$-a.s., $(X_n\cdot e_1)/n\to v$ for
some $v\ge 0$. Due to (\ref{ski}) we even have $v>0$. Since obviously
$P_{0,\om_1}$-a.s.\ $X'_n/n\to 0$ , this yields the statement
(\ref{hin}).  }
\end{ex}

\textbf{Open questions.} We have already mentioned that not much is
known about positively and negatively ERW in $d\ge
2$. Below we shall discuss $d=1$. It might be also interesting to
consider ERW on strips $\Z\times \{0,1,\dots,L\}$, $L\ge 1$.

\textit{(a) Recurrent regime.} For the case of non-negative cookies
and $\delta<1$, D.\,Dol\-go\-pyat has shown, \cite{Do08}, (the case of
strips was also considered) that, under some assumptions,
$\frac{X_{[nt]}}{\sqrt{n}}$ converges in law to the unique pathwise solution $W(t)$ (see \cite{CD99})
of the equation
\[W(t)=B(t)+\delta\left(\max_{[0,t]} W(t)-\min_{[0,t]} W(t)\right), \]
where $B(t)$ is the Brownian motion with variance $t$. Can this
result be extended to positively and negatively excited random
walks? What happens in the case $|\delta|=1$?

\textit{(b) Transient regime with zero linear speed.} A.-L.\ Basdevant
and A.\ Singh obtained in \cite{BS07} for non-negative (deterministic)
cookie environments and $1<\delta\le 2$ the following results.
  \begin{itemize}
  \item [1.] If $\delta\in (1,2)$ then $\frac{X_n}{n^{\delta/2}}$
    converges in law to a random variable 
    $S^{-\delta/2}$, where $S$ is a positive strictly stable random variable with index $\delta/2$, i.e.\ with
    Laplace transform $E[e^{-\la S}]=e^{-c\la^{\delta/2}}$ for some $c>0$.
  \item [2.] If $\delta=2$ then $\frac{X_n}{n/\log n}$ converges in
    probability to a positive constant.
  \end{itemize}
  Their proof is based on the study of branching processes with
  migration but uses the assumption that all cookies are non-negative.
  The same result with essentially the same proof might hold in the
  more general setting studied in the current paper.  Is there a
  result for $\delta=2$ similar to (ii) of \cite{KKS75}?

\textit{(c) Transient regime with positive linear speed.} We do not
  know whether our condition $\delta>4$ for the validity of the central limit
  theorem is optimal. How does the process scale for $\delta\in(2,4]$?
  Is the behavior similar to (iii) and (iv) of
  \cite{KKS75}?

\bibliographystyle{amsalpha}
\vspace*{5mm}
{\sc \small
\begin{tabular}{ll}
Department of Mathematics& Mathematisches Institut\\
Baruch College&Universit\"at T\"ubingen\\
One Bernard Baruch Way, Box B6-230&Auf der Morgenstelle 10\\
New York, NY 10010, U.S.A.&72076 T\"ubingen, Germany\\
{\rm elena.kosygina@baruch.cuny.edu} & {\rm martin.zerner@uni-tuebingen.de} 
\end{tabular}}
\end{document}